\documentclass{article}

\usepackage[utf8]{inputenc}
\usepackage{geometry, mathtools, amsmath, amsfonts, amsthm, cases, thmtools, tabularx, appendix, xcolor, bbm, titling, bm}
\usepackage[maxbibnames=9]{biblatex}
\usepackage{tikz,lipsum,lmodern}
\usepackage[most]{tcolorbox}

\renewbibmacro{in:}{%
  \ifentrytype{article}
    {}
    {\bibstring{in}%
     \printunit{\intitlepunct}}}
\DeclareFieldFormat{pages}{#1}
\newtheorem{theorem}{Theorem}[section]
\newtheorem{definition}[theorem]{Definition}
\newtheorem{proposition}[theorem]{Proposition}
\newtheorem{lemma}[theorem]{Lemma}
\newtheorem{corollary}[theorem]{Corollary}
\newtheorem{remark}[theorem]{Remark}

\DeclareMathOperator*{\esssup}{ess\,sup}
\title{Hard congestion limit of the dissipative Aw-Rascle system with a polynomial offset function}
\author{Muhammed Ali Mehmood\thanks{Imperial College London, London, United Kingdom; muhammed.mehmood21@imperial.ac.uk}}
\date{\today}

\begin{document}
\maketitle
\thispagestyle{empty}

\begin{abstract}
We study the Aw-Rascle system in a one-dimensional domain with periodic boundary conditions, where the offset function is replaced by the gradient of the function $\rho_{n}^{\gamma}$, where $\gamma \to \infty$. The resulting system resembles the 1D pressureless compressible Navier-Stokes system with a vanishing viscosity coefficient in the momentum equation and can be used to model traffic and suspension flows. We first prove the existence of a unique global-in-time 
classical solution for fixed $n$. Unlike the previous result for this system, we obtain global existence without needing to add any approximation terms to the system. This is by virtue of a $n-$uniform lower bound on the density which is attained by carrying out a maximum-principle argument on a suitable potential, $W_{n} = \rho_{n}^{-1}\partial_{x}w_{n}$. Then we prove the convergence to a weak solution of a hybrid free-congested system as $n \to \infty$, which is known as the hard-congestion model.
\end{abstract} \vspace{1ex}

\section{Introduction}
\subsection{The Aw-Rascle system}
This paper aims to study a singular limit pertaining to the following generalisation of the Aw-Rascle \cite{aw2000resurrection} and Zhang \cite{zhang_non-equilibrium_2002} system:
\begin{subequations} 
\newcommand{\mystrut}{\vphantom{\pder{}{}}}
\begin{numcases}{}
    \partial_{t} \rho_{n} + \partial_{x}(\rho_{n} u_{n}) = 0, &\text{ in } $\Omega \times [0, \infty), $ \label{AR-1} \\[1ex]
    \partial_{t} (\rho_{n} w_{n}) + \partial_{x} (\rho_{n} w_{n} u_{n}) = 0, &\text{ in } $\Omega \times [0, \infty). $ \label{AR-2}
\end{numcases}
\end{subequations}
We take $\Omega = \mathbb{T}$ to be the one-dimensional torus, which we identify with $[0,1]$. This system is to be solved for $\rho_{n}$ and $u_{n}$ which represent the density and the (actual) velocity respectively. The quantity $w_{n}$ that appears in the above system is known as the desired velocity and differs from the actual velocity $u_{n}$ by a cost (or offset) function $C_{n}$, which typically depends on the density $\rho_{n}$. More precisely, $w_{n}$ is defined by the relation
\begin{equation} \label{wucost} \tag{1c}
    w_{n} = u_{n} + C_{n}(\rho_{n}).
\end{equation} The cost function may be interpreted as a quantity which measures the difficulty of moving in a certain direction. In this paper, we consider the case where our cost function is the gradient of a function $p_{n}(\rho_{n})$, i.e.
    \begin{equation*}
    C_{n}(\rho_{n}) = \partial_{x} p_{n}(\rho_{n}),
    \end{equation*}
    where \begin{equation}
        p_{n}(\rho_{n}) = \rho_{n}^{\gamma_{n}}.    \label{pressure-n}
    \end{equation}
Here, $\left\{\gamma_{n}\right\}_{n = 1}^{\infty}$ is a sequence of positive real numbers satisfying $\gamma_{n} \to \infty$ as $n \to \infty$. We note that the singularity of this function only appears when considering the limit as $n \to \infty$; in the region where $\rho_{n} > 1$, we have $p_{n}(\rho_{n}) \to +\infty$ as $n \to \infty$. Our goal is to determine the existence of a solution to the above system for $n$ fixed before investigating the limiting behaviour of the solution as $n \to \infty$. 
\subsection{Background and motivation}
The derivation and mathematical analysis of traffic flow models has been an active area of research over the past hundred years, with some notable early works devoted to hydrodynamic models being \cite{lighthill_kinematic_1955, daganzo_requiem_1995, zhang_non-equilibrium_2002, gazis1961nonlinear}. The Aw-Rascle model, which is given by \eqref{AR-1}-\eqref{wucost} with a scalar offset $C(\rho) = \rho^{\gamma},~\gamma>1$,  is one of the most famous examples of such models. The derivation of this model from the ‘Follow-The-Leader' microscopic model for one-lane traffic can be seen in \cite{derivationFTLAR}. Much analysis has been devoted to the Aw-Rascle model with different choices of offset functions and domains over the past twenty years. We refer to \cite{HCL, chaudhuri_analysis_2022, locallycosntrainedAR, CHAPLYGINpressure, resonanceAr} for some recent examples. It is known that although the model proposed by Aw and Rascle in \cite{aw2000resurrection} was an improvement upon previous iterations of traffic flow models, the system still exhibits some non-physical behaviour. For instance, solutions are assumed to adhere to the maximal density constraint $\rho < \overline{\rho}$ initially but proceed to violate it in finite time. 
 In order to overcome these unrealistic behaviours, Berthelin et al \cite{berth2008degond} suggested to study the asymptotic limit of the system \eqref{AR-1}-\eqref{AR-2} accompanied by a singular scalar offset function $p_{n}(\rho_{n})$ instead of $\rho^{\gamma}$. A recent paper by Chaudhuri et al \cite{HCL} followed this suggestion and studied the case where the offset function is equal to the gradient of a singular function, i.e.
\begin{equation*}
  w_{\epsilon} = u_{\epsilon} + \partial_{x}p_{\epsilon}(\rho_{\epsilon}),~\hspace{5pt}p_{\epsilon}(\rho_{\epsilon}) =  \epsilon \frac{\rho_{\epsilon}^{\gamma}}{(1-\rho_{\epsilon})^{\beta}},
\end{equation*}in a one-dimensional domain with periodic boundary conditions. With this choice, the density is now prevented from surpassing a maximal threshold $\overline{\rho} \equiv 1$. The authors go on to study the asymptotic limit $\epsilon \to 0$ which is known as the 'hard-congestion limit', and establish the existence of solutions $(\rho, u, \pi)$ to the limiting system
\begin{subequations} \label{HCL-LIMIT}
\newcommand{\mystrut}{\vphantom{\pder{}{}}}
\begin{numcases}{}
    \partial_{t} \rho + \partial_{x} (\rho u) = 0, \label{HCL-L1}   \\[1ex]
    \partial_{t} (\rho  u + \partial_{x}\pi) + \partial_{x}((\rho u + \partial_{x}\pi)u)  = 0,  \label{HCL-L2} \\[1ex]
    0 \le \rho \le 1,~ (1-\rho)\pi = 0,~ \pi \ge 0, \label{HCL-L3}
\end{numcases}
\end{subequations} where $\pi$ is the limit of some singular function of $\rho_{\epsilon}$. This is known as the hard congestion model. The interest in this system, which is an example of a free-congested system, emerges from the observation that it repairs the aforementioned issues with the classical Aw-Rascle model. In particular, from \eqref{HCL-L3} we notice that the potential $\pi$ obtained in the singular limit is zero except when the density is maximal, where it acts similar to a Dirac measure. This reflects an important characteristic of traffic flow, which is that drivers do not typically slow down unless there is congestion. The authors of \cite{HCL} prove two main results for their system; the existence of a unique global strong solution for $\epsilon$ fixed, and the existence of a subsequence converging to a solution of the hard-congestion model as $\epsilon \to 0$. 

Our paper builds upon the work of Chaudhuri et al \cite{HCL} and attempts to prove analogous results for the case where the singular offset function is of the form $p_{n}(\rho_{n}) = \rho_{n}^{\gamma_{n}}$. One particular piece of motivation for this problem is that our form of offset function makes it easier to perform numerical simulations (such as in \cite{mohammadian_improved_2018,andreianov_solutions_2016, sheng_concentration_2022}) and investigate the behaviour of solutions than the model in \cite{HCL}, for example. Additionally, our results could be used in conjunction with those in \cite{HCL} in order to analyse the Aw-Rascle system with a more complex offset function and/or in the multi-dimensional case. One example of a multi-dimensional study of the Aw-Rascle system can be seen in a recent paper by Chaudhuri, Gwiazda and Zatorska \cite{chaudhuri_analysis_2022}. 

For our analysis it will be useful to note that if we fix $n$ with this choice of cost function, the system \eqref{AR-1}-\eqref{AR-2} can be formally rewritten as the one-dimensional compressible pressureless Navier-Stokes equations
\begin{subequations} \label{A}
\newcommand{\mystrut}{\vphantom{\pder{}{}}}
\begin{numcases}{}
    \partial_{t} \rho_{n} + \partial_{x} (\rho_{n} u_{n}) = 0, &\text{ in } $\mathbb{T} \times [0, \infty), $  \label{a1} \\[1ex]
    \partial_{t} (\rho_{n} u_{n}) + \partial_{x}(\rho_{n} u_{n}^{2}) - \partial_{x}(\lambda_{n}(\rho_{n})\partial_{x}u_{n}) = 0, &\text{ in } $\mathbb{T} \times [0, \infty),$  \label{a2}
\end{numcases}
\end{subequations}where $\rho_{n}, u_{n} : \mathbb{T} \times [0, \infty) \to \mathbb{R}$ are to be found and \begin{equation*}
    \lambda_{n}(\rho_{n}) = \rho_{n}^{2}p_{n}'(\rho_{n}), ~p_{n}(\rho_{n}) = \rho_{n}^{\gamma_{n}}, ~\gamma_{n} \in (0, \infty).
\end{equation*}
The systems \eqref{AR-1}-\eqref{AR-2} and \eqref{a1}-\eqref{a2} are equivalent for sufficiently regular solutions, and in particular for the class of regular solutions which we will consider. Interestingly, a similar approximation of a two-phase system was carried out by Lions and Masmoudi \cite{Lions1999}, where the authors consider a compressible Navier-Stokes system with a pressure $\pi = \rho^{\gamma}$ and study the limit $\gamma \to \infty$. The presence of a constant viscosity coefficient allows the authors of \cite{Lions1999} to control the gradient of the velocity $\nabla\mathbf{u}$ which gives way to a crucial uniform bound on the singular pressure. Although our system is pressureless, we note that the potential $\pi_{n}$ which is defined through the relation
\begin{equation} 
    \pi_{n}'(\rho_{n}) := \rho_{n}p'_{n}(\rho_{n}) = \gamma_{n} \rho_{n}^{\gamma_{n}},   \label{pi1}
\end{equation} plays a similar role to the pressure in the classical compressible Navier-Stokes model. We need to control this term in order to obtain the switching relation \eqref{HCL-L3} in the limit system. However, the presence of a degenerate viscosity coefficient in our system means that we cannot bound the potential in the same way that the pressure was bounded by the authors of \cite{Lions1999}. For this reason, we need to carry out an improved potential estimate.
\subsection{Main results}
In this paper we adopt the Bochner space notation $X_{t}Y_{x} := X(0,T; Y(\mathbb{T}))$ for appropriate function spaces $X$ and $Y$. We first provide a precise definition of regular solutions to the system 
\eqref{a1}-\eqref{a2}, which are also classical.
\begin{definition}[Global regular solution] \label{defstrongsoln}
    Suppose $n \in \mathbb{Z}^{+}$ is fixed, $T>0$ and $p_{n}$ is given by \eqref{pressure-n}. Assume further that $(\rho_{n}^{0}, u_{n}^{0}) \in H^{4}(\mathbb{T}) \times H^{4}(\mathbb{T})$ and $0 < \rho_{n}^{0}(\cdot)$. The pair $(\rho_{n}, u_{n})$ is called a regular solution to \eqref{a1}-\eqref{a2} on $[0,T]$ if \begin{equation*}
        \rho_{n} \in C(0,T; H^{4}(\mathbb{T})), ~u_{n} \in C(0,T; H^{4}(\mathbb{T})) \cap L^{2}(0, T ; H^{5}(\mathbb{T))},
    \end{equation*} and $(\rho_{n}, u_{n})$ satisfy \eqref{a1}-\eqref{a2} in $\mathbb{T} \times [0,T]$. The pair $(\rho_{n}, u_{n})$ is known as a global regular solution to \eqref{a1}-\eqref{a2} if it is a regular solution on $[0,T]$ for any $T>0$.
\end{definition}
We will prove two main results. Firstly, we will assert the existence of a unique global regular solution in the following result:
\begin{theorem}[Global existence of a unique regular solution for fixed $n$] \label{globalexistence}
    Assume $~(\rho_n^{0}, u_n^{0}) \in H^{4}(\mathbb{T}) \times  H^{4}(\mathbb{T})$ and that $n \in \mathbb{Z}^{+}$ is fixed. Further assume that $ 0 < \rho_{n}^{0}(x)$ for $x \in \mathbb{T}$. Then there exists a unique pair $(\rho_{n}, u_{n})$  with initial data $(\rho_n^{0}, u_n^{0})$ which is a global regular solution to \eqref{a1}-\eqref{a2} in the sense of Definition \ref{defstrongsoln}.
\end{theorem}
\begin{remark}
The $H^{4}$ regularity on the initial data mentioned in Definition \ref{defstrongsoln} and Theorem \ref{globalexistence} is required in order for us to obtain a lower bound on the density. Our strategy demands that on the level of local-in-time solutions the desired velocity $w_{n} \in C_{t}H^{3}_{x}$, which itself requires $\rho_{n}, u_{n} \in C_{t}H^{4}_{x}$. To work with such solutions using Theorem \ref{localexistence} (local existence of solutions), we must take $(\rho_{n}^{0}, u_{n}^{0}) \in H^{4}(\mathbb{T}) \times H^{4}(\mathbb{T})$.
\end{remark}
We then assert the existence of weak solutions to the limiting system in the following theorem:
\begin{theorem}[Global existence of a weak solution to the hard-congestion model] \label{limitexistence} Assume $~(\rho_n^{0}, u_n^{0}) \in H^{4}(\mathbb{T}) \times  H^{4}(\mathbb{T})$ and the existence of constants $C, \alpha, \hat{\rho} > 0$ independent of $n$ such that
\begin{align}
 &0 < \rho_{n}^{0}(x) \le 1 + \frac{1}{\gamma_{n}},~~\forall x \in \mathbb{T}, \label{thm2RHObounds} \\[1ex]   &0 < \alpha \le |\mathbb{T}|^{-1}\int_{\mathbb{T}}\rho_{n}^{0}(x)~dx \le \hat{\rho} < 1, \label{thm2meanvalue} \\[1ex]
&  \|\sqrt{\rho_{n}^{0}}w_{n}^{0}\|_{L^{2}_{x}} + \left\|\frac{\partial_{x}w_{n}^{0}}{\sqrt{\rho_{n}^{0}}}\right\|_{L^{2}_{x}}  \le C. \label{thm2wn}
    \end{align} Then, the solution $(\rho_{n}, u_{n})$ established in Theorem \ref{globalexistence} satisfies the following uniform bounds for $\gamma_{n} > 1$ and $\tau \in (0,T]$:
\begin{align}
    &\|w_{n}\|_{L^{\infty}_{t}W^{1,\frac{4}{3}}_{x}} \le C, \label{thm2wnuniformbound}\\[1ex]
    &\|\pi_{n}\|_{L^{1}_{t,x}} + \|\partial_{x}\pi_{n}\|_{L^{2}_{t,x}} \le C, \\[1ex]
    &\|\pi_{n}\|_{L^{\infty}([\tau, T]; L^{1}(\mathbb{T}))} \le C. \label{thm2pilocal}
\end{align}
    where $C>0$ is independent of $n$. Suppose additionally that
    \begin{align}
        &\rho_{n}^{0} \rightharpoonup \rho^{0} \text{ weakly in } L^{2}(\mathbb{T}), \label{idconvergence1} \\[1ex]
        &\rho_{n}^{0}w_{n}^{0} \rightharpoonup \rho^{0}w^{0} \text{ weakly in } L^{2}(\mathbb{T}). \label{idconvergence2}
    \end{align} Then there exists a subsequence $(\rho_{n}, w_{n}, \pi_{n})$ of solutions to \eqref{a1}-\eqref{a2} with initial data $(\rho_{n}^{0}, w_{n}^{0})$ which converges to $(\rho, w, \pi)$ solving:
    \begin{subequations} 
\newcommand{\mystrut}{\vphantom{\pder{}{}}}
\begin{numcases}{}  \label{HCM-WF1}
   -  \int_{0}^{t} \int_{\mathbb{T}} \rho \partial_{t} \phi ~dxds + \int_{\mathbb{T}} \rho(x,t)\phi(x,t) - \rho^{0}(x)\phi(x,0)~dx + \int_{0}^{t} \int_{\mathbb{T}} \partial_{x} \pi \partial_{x} \phi ~dxds= \int_{0}^{t} \int_{\mathbb{T}} \rho w \partial_{x} \phi ~dxds, \\[1ex]
   -  \int_{0}^{t} \int_{\mathbb{T}} \rho w \partial_{t}\phi ~dxds + \int_{\mathbb{T}}\rho w \phi(x,t) - \rho^{0} w^{0}(x) \phi (x,0)~dx +  \int_{0}^{t} \int_{\mathbb{T}} (w \partial_{x} \pi - \rho w^{2})\partial_{x}\phi ~dxds =  0. \label{HCM-WF2} \\[1ex]
   0 \le  \rho \le 1,~(1-\rho)\pi = 0,~ \pi \ge 0. \label{HCM-WF3}
\end{numcases}
\end{subequations} with initial data $(\rho^{0}, w^{0})$. Additionally, our solution possesses the following regularity for $p \in [1,\infty)$ and $\tau \in (0,T)$:
    \begin{align} \label{solnrhoreg}
        &\rho \in C_{weak}(0,T; L^{p}(\mathbb{T})) \cap L^{\infty}( (0,T) \times \mathbb{T}), \\[1ex] \label{solnwreg}
        &w \in C(0,T; W^{1,\frac{4}{3}}(\mathbb{T})), \\[1ex]
        &\pi \in \mathcal{M}((0,T) \times \mathbb{T}) \cap L^{\infty}([\tau, T]; L^{1}(\mathbb{T})) \cap L^{2}([\tau, T]; L^{2}(\mathbb{T})). \label{solnpireg}
    \end{align}
\end{theorem}
\begin{remark}
    The upper bound in assumption \eqref{thm2RHObounds} and the bound on $\sqrt{\rho_{n}^{0}}w_{n}^{0}$ in \eqref{thm2wn} are required to obtain important uniform bounds from the additional energy estimate (Lemma \ref{bdentropy}), and also to obtain the bound \eqref{thm2pilocal}. The second bound appearing in \eqref{thm2wn} is a crucial part of our argument since it allows us to uniformly bound $\partial_{x}w_{n}$ and even $\partial_{t}w_{n}$ in $L^{p}_{t,x}$ spaces. The lower bound on the mean value in \eqref{thm2meanvalue} is used to apply the Poincare inequality which gives us a uniform $L^{\infty}_{t}L^{1}_{x}$ bound on $w_{n}$. Lastly, the upper bound assumption in \eqref{thm2meanvalue} is used to derive the bound \eqref{thm2pilocal}.
\end{remark}
\begin{remark}
    If we additionally assume that there exist constants $c, C > 0$ independent of $n$ such that \begin{align*}
        0 < c \le \inf_{\mathbb{T}} \rho_{n}^{0}, ~~ \esssup_{\mathbb{T}} \frac{\partial_{x}w_{n}^{0}}{\rho_{n}^{0}} \le C,
    \end{align*}
     then as a consequence of the proof of Theorem \ref{globalexistence} we can show that the sequence of densities $\rho_{n}$ is uniformly bounded away from zero, i.e. $0 < c \le \rho_{n}$. As a result the density for the limiting system inherits the same lower bound.
\end{remark}
\begin{remark}
     Note that the limiting system \eqref{HCM-WF1}-\eqref{HCM-WF3} is slightly different in appearance to that of \cite{HCL}, since our system is written in terms of the desired velocity $w$ rather than the actual velocity $u$. One can also notice the assumptions on the initial data \eqref{thm2RHObounds}-\eqref{thm2wn} are weaker than those appearing in \cite{HCL}, in the sense that we do not assume any control over $\partial_{x}\pi_{n}^{0}$ or the singular quantity $\lambda_{n}(\rho_{n}^{0})\partial_{x}u_{n}^{0}$. This is made possible through our decision to work with the $w-$formulation rather than the $u-$formulation. It is also worthwhile to mention that if we define the function $u := w - \overline{\rho^{-1}\partial_{x} \pi}$ where $\overline{\rho^{-1}\partial_{x} \pi}$ is the weak limit of $\rho_{n}^{-1}\partial_{x} \pi_{n}$, the limiting system is equivalent to \eqref{AR-1}-\eqref{AR-2} (without the index $n$) in the distributional sense, accompanied by the condition \eqref{HCM-WF3}. 
\end{remark}
The global existence result given by Theorem \ref{globalexistence} is an improvement upon the work of \cite{HCL} where the authors added an approximation term to the system in order to derive an ($\epsilon$ dependent) lower bound on the density. Our derivation of the lower bound involves identifying a suitable potential for the Aw-Rascle system $W_{n}:=\partial_{x}w_{n}/\rho_{n}$ and carrying out a maximum-principle argument to deduce that the maximum of $W_{n}$ is a decreasing function over time. This approach takes inspiration from \cite{Constantin_2020}, \cite{Burtea_2020} and shows that an approximation term is not necessary in order to obtain a (uniform) lower bound on the density for a generalised Aw-Rascle system.  In fact, this argument also gives us control over the quantity $\partial_{x}w_{n}$ which is essential for the proof of Theorem \ref{limitexistence}. Another difference between our paper and previous works on the Aw-Rascle system \cite{HCL, chaudhuri_analysis_2022} is that we pass to the limit in the so-called '$w-$formulation'
\begin{subequations} 
\newcommand{\mystrut}{\vphantom{\pder{}{}}}
\begin{numcases}{} \label{limsystem1}
    \partial_{t} \rho_{n} + \partial_{x}(\rho_{n} w_{n} - \partial_{x}\pi_{n}) = 0, \\[1ex] 
    \partial_{t} (\rho_{n} w_{n}) + \partial_{x} (\rho_{n} w_{n}^{2}) - \partial_{x} (w_{n} \partial_{x} \pi_{n}) = 0, \label{limsystem2}
\end{numcases}
\end{subequations}
rather than the '$u-$formulation' seen in \eqref{a1}-\eqref{a2}. The existence of a lower bound on the density (obtained in the proof of Theorem \ref{globalexistence}) makes this formulation more convenient, since we may divide by $\rho_{n}$ in the momentum equation \eqref{limsystem2} and directly obtain energy estimates on $\partial_{x}w_{n}$. In \cite{HCL} where the $u-$formulation was used instead, the authors were required to assume that the singular quantity $\lambda_{n}\partial_{x}u_{n}$ is uniformly bounded at time $t=0$ to complete the limit passage. We do not need to make such an assumption by working with $w_{n}$. It is worthwhile to note that in the distributional sense, our limiting system in $w-$form is (compare this with \eqref{HCL-L1}-\eqref{HCL-L3}):
\begin{subequations} 
\newcommand{\mystrut}{\vphantom{\pder{}{}}}
\begin{numcases}{} 
    \partial_{t} \rho + \partial_{x}(\rho w - \partial_{x} \pi) = 0, \label{w1-intro} \\[1ex]
    \partial_{t} (\rho w) + \partial_{x} (\rho w^{2}) - \partial_{x} (w \partial_{x} \pi) = 0, \label{w2-intro} \\[1ex]
    0 \le  \rho \le 1,~(1-\rho)\pi = 0,~ \pi \ge 0. \label{w3-intro}
\end{numcases}
\end{subequations}
\subsection{Overview of the paper}
The paper is comprised as follows. In Sections 2.1/2.2 we make note of a local existence result and three key energy estimates. These allow us to obtain an upper bound of the density after which we prove a 'blow-up' lemma in Section 2.4, which is analogous to what can be seen in Theorem 1.1 of \cite{Constantin_2020}. The blow-up lemma tells us that provided the density is positive on a domain $\mathbb{T} \times [0,T)$, our local solution can be extended past time $T$. In Section 2.5 we use this result to carry out a maximum-principle argument on the potential $W_{n}:=\partial_{x}w_{n}/\rho_{n}$ in order to show that the density $\rho_{n}$ is bounded from below on $\mathbb{T} \times [0,T]$. This implies that our solution exists globally, thanks to the blow-up lemma.

The second half of the paper (Sections 3 and 4) is dedicated towards the limit passage. We take advantage of the transport equation satisfied by $w_{n}$ and $W_{n}$ to uniformly bound $w_{n}$ in $L^{\infty}_{t}W^{1,p}_{x}$, which eventually leads to strong convergence on $w_{n}$. This is particularly useful when passing to the limit in the non-linear terms appearing in \eqref{limsystem2}. A key obstacle standing in our way at this point is acquiring a bound on $\pi_{n}$ and $\partial_{x} \pi_{n}$. Using the same strategy as the authors of \cite{HCL} which is to test the momentum equation with an antiderivative of the density will not work for us, since our assumptions on the initial data are considerably weaker, and so we have access to fewer uniform bounds. In particular, we have no $L^{p}_{t,x}$ estimates on the momentum $\rho_{n}u_{n}$ at this stage. This prevents us from bounding many of the terms which would appear in the momentum equation after testing. To overcome this, we take a very intricate choice of test function in the continuity equation which gives us a $L^{1}_{t,x}$ bound on $\pi_{n}$. This delicate estimate is what allows us to converge towards a weak solution even with our relatively weak assumptions on the initial data. The last uniform bound we collect is a slightly stronger local-in-time estimate for $\pi_{n}$ that is needed to derive the switching relation $(1-\rho)\pi=0$ for the limiting system. The uniform bounds which we obtain for $\pi_{n}$ are weaker than those appearing in \cite{HCL}, which is to be expected. Nonetheless we show that it is still possible to obtain a weak solution to the limiting system.
\section{Existence of a unique global regular solution for fixed $n$}
\subsection{Local existence}
The existence of a unique local regular solution for fixed $n$ can be shown in a very similar way to Proposition B.1 in \cite{Constantin_2020}. The full details of the proof are omitted from this paper.
\begin{theorem}[Existence of a unique local regular solution] \label{localexistence}    Assume $~(\rho_{0}, u_{0}) \in H^{k}(\mathbb{T}) \times  H^{k}(\mathbb{T})$, $k \ge 1$ and that $ \displaystyle r_{0} := \min_{x \in \mathbb{T}} \rho_{0} > 0$. Then there exists $T_{0}>0$ (depending solely on $r_{0}$ and the initial data) and a unique solution $(\rho, u)$ to \eqref{A} for $t \in [0,T_{0}]$ with initial data $(\rho_{0}, u_{0})$ such that \ \begin{equation*}
        \rho \in C(0,T_{0}; H^{k}(\mathbb{T})), ~u \in C(0,T_{0}; H^{k}(\mathbb{T})) \cap L^{2}(0, T_{0} ; H^{k+1}(\mathbb{T))}.
    \end{equation*} Additionally, we have that $\rho(x,t) \ge \frac{r_{0}}{2} \text{ for each} ~(x,t) \in \mathbb{T} \times [0,T_{0}]$.
\end{theorem}
Starting with some initial data  $(\rho_{n}^{0}, u_{n}^{0}) \in H^{4}(\mathbb{T}) \times H^{4}(\mathbb{T})$, we may take $k=4$ in Theorem \ref{localexistence} to obtain the existence of a solution $(\rho_{n}, u_{n})$ to \eqref{a1}-\eqref{a2} on $[0,T_{0}]$ for some $T_{0} > 0$, where $n \in \mathbb{Z}^{+}$ is fixed. Let us now denote by $T^{*}$ the maximal time of existence of our solution $(\rho_{n}, u_{n})$. The purpose of this section is show that our solution can be extended to one that is defined globally in time. 
\subsection{Energy estimates} \label{sectionenergy}
In this subsection, we assume that $(\rho_{n},u_{n})$ is a regular solution to \eqref{A} on some time interval $[0,T]$ and that $\rho_{n} \ge 0$ on $\mathbb{T} \times [0,T]$.  Our aim is to establish three uniform in time energy estimates which we will need in order to extend our solution to one that is defined globally in time. 
Our first two estimates are classical; the first is a consequence of the conservation of mass.
\begin{lemma}[Conservation of mass]\label{consofmass}
    Assume that $(\rho_{n}, u_{n})$ is a regular solution to \eqref{a1}-\eqref{a2} on the time interval $[0,T]$ and additionally that $\rho_{n}(\cdot, t) \ge 0$ on this interval. Then, 
    \begin{equation} \vspace{-10pt} \label{E1} 
        \|\rho_{n}(t)\|_{L^{1}_{x}} = \underbrace{\|\rho_{n}^{0}\|_{L^{1}_{x}}}_{=:E_{n}^{0}},
    \end{equation} for all $t \in [0,T]$.
\end{lemma}
The next energy estimate is derived from the momentum equation upon multiplying by $u_{n}$ and integrating by parts.
\begin{lemma}[Basic energy]\label{basicenergy}
    Assume that $(\rho_{n}, u_{n})$ is a regular solution to \eqref{A} on the time interval $[0,T]$ and additionally that $\rho_{n}(\cdot, t) \ge 0$ on this interval. Then,
    \begin{equation} \vspace{-10pt}
        \label{E2} 
        \| \sqrt{\rho_{n}}u_{n}\|^{2}_{L^{\infty}_{t}L^{2}_{x}} + 2\| \sqrt{\lambda_{n}(\rho_{n})}\partial_{x}u_{n}\|^{2}_{L^{2}_{t}L^{2}_{x}} = \underbrace{\| \sqrt{\rho_{n}^{0}}u_{n}^{0}\|^{2}_{L^{2}_{x}}}_{=: E_{n}^{1}}
    \end{equation} for all $t \in [0,T]$.
\end{lemma}
The final estimate in this section provides us with a bound on $\partial_{x} p_{n}(\rho_{n})$. Here, it is convenient to introduce the notation
    \begin{equation}
        H_{n}(\rho_{n}) := \frac{1}{\gamma_{n}+1}\rho_{n}^{\gamma_{n}+1}. \label{Hnform}
    \end{equation}
\begin{lemma}[Additional energy]\label{bdentropy}
     Assume that $(\rho_{n}, u_{n})$ is a regular solution to \eqref{A} on the time interval $[0,T]$ and additionally that $\rho_{n}(\cdot, t) \ge 0$ on this interval. Then,
    \begin{equation}
        \label{E3} \vspace{-10pt} \| \sqrt{\rho_{n}}w_{n}\|^{2}_{L^{\infty}_{t}L^{2}_{x}} + \|H_{n}(\rho_{n})\|_{L^{\infty}_{t}L^{1}_{x}} + \frac{1}{2}\| \sqrt{\rho_{n}} \partial_{x} p_{n}(\rho_{n})\|^{2}_{L^{2}_{t}L^{2}_{x}} \le \left(T+2\right)\underbrace{\left( \|\sqrt{\rho_{n}^{0}}w_{n}^{0}\|^{2}_{L^{2}_{x}} + \|H_{n}(\rho_{n}^{0})\|_{L^{1}_{x}} \right)}_{=: E_{n}^{2}}.
    \end{equation}
\end{lemma}
\begin{proof}
    For regular solutions in the sense of Definition~\ref{defstrongsoln}, \eqref{a2} is equivalent to \begin{equation*}
         \partial_{t} (\rho_{n}w_{n}) + \partial_{x}(\rho_{n} w_{n}u_{n}) = 0, \hspace{4pt} \text{in } \mathbb{T} \times [0,T]. 
    \end{equation*} Multiplying this equation by $w_{n}$ and integrating by parts, it is straightforward to show that
    \begin{equation}
        \|\sqrt{\rho_{n}} w_{n}\|_{L^{\infty}_{t}L^{2}_{x}} \le \| \sqrt{\rho_{n}^{0}} w^{0}_{n}\|_{L^{2}_{x}}. \label{bde1}
    \end{equation} Substituting $w_{n} = u_{n} + \partial_{x} p_{n}(\rho_{n})$ into the mass equation, we get
    \begin{equation*}
        \partial_{t} \rho_{n} + \partial_{x} (\rho_{n}w_{n}) - \partial_{x} \left( \rho_{n} \partial_{x} p_{n}(\rho_{n})\right) = 0,
    \end{equation*}in $\mathbb{T} \times [0,T]$. Multiplying this equation by $H_{n}'(\rho_{n}) = p_{n}(\rho_{n})$, integrating over $\mathbb{T} \times [0,t]$ where $t  \in [0,T]$ and using integration by parts leads to
    \begin{align}
        \int_{\mathbb{T}} H_{n}(\rho_{n}(t)) - H_{n}(\rho_{n}^{0})~dx + \| \sqrt{\rho_{n}} \partial_{x} p_{n}(\rho_{n})\|^{2}_{L^{2}_{t}L^{2}_{x}}  = \int_{0}^{t} \int_{\mathbb{T}} (\partial_{x}p_{n}(\rho_{n})) \rho_{n}w_{n}  ~dxds. \label{bde1.5}
    \end{align} Using Young's inequality,
    \begin{align*}
         \int_{0}^{t} \int_{\mathbb{T}} (\partial_{x}p_{n}(\rho_{n})) \rho_{n}w_{n}  ~dxds \le \frac{1}{2}\|\sqrt{\rho_{n}} \partial_{x}p_{n}(\rho_{n})\|^{2}_{L^{2}_{t}L^{2}_{x}} + \frac{1}{2} \| \sqrt{\rho_{n}}w_{n}\|^{2}_{L^{2}_{t}L^{2}_{x}}.
    \end{align*} Thus, returning to \eqref{bde1.5} and taking the essential supremum over all $t \in [0,T]$, we have
    \begin{equation*}
        \|H_{n}(\rho_{n})\|_{L^{\infty}_{t}L^{1}_{x}} + \frac{1}{2}\| \sqrt{\rho_{n}} \partial_{x} p_{n}(\rho_{n})\|^{2}_{L^{2}_{t}L^{2}_{x}} \le \|H_{n}(\rho_{n}^{0})\|_{L^{1}_{x}} + \frac{1}{2} \| \sqrt{\rho_{n}}w_{n}\|^{2}_{L^{2}_{t}L^{2}_{x}}. 
    \end{equation*} Adding \eqref{bde1} to both sides of this inequality and using the estimate $\|\sqrt{\rho_{n}}w_{n}\|_{L^{2}_{t,x}} \le T\|\sqrt{\rho_{n}}w_{n}\|_{L^{\infty}_{t}L^{2}_{x}}$ gives us the final result.
\end{proof}
\subsection{Estimating the density from above} \label{sectiondensityub}
We now obtain an upper bound for the density.
\begin{lemma} \label{upperboundrho}
    Suppose $(\rho_{n}, u_{n})$ is a regular solution to \eqref{a1}-\eqref{a2} on $[0,T]$ and additionally that $\rho_{n} \ge 0$ on $[0,T]$. Then,
    \begin{equation*}
        \rho_{n}(t,x) \le  \left( C(\gamma_{n}+1)(T+5)E_{n} \right)^{\frac{1}{\gamma_{n}+1}} =: \overline{\rho_{n}}, ~\text{ in } \mathbb{T} \times [0,T],
    \end{equation*} where $C>0$ is independent of $n$ and \begin{equation*}
        E_{n} := E_{n}^{0} + E_{n}^{1} + E_{n}^{2} = \|\rho_{n}^{0}\|_{L^{1}_{x}} + \|\sqrt{\rho_{n}^{0}}u_{n}^{0}\|^{2}_{L^{2}_{x}} + (\|\sqrt{\rho_{n}^{0}}w_{n}^{0}\|^{2}_{L^{2}_{x}} + \|H_{n}(\rho_{n}^{0})\|_{L^{1}_{x}}).
    \end{equation*} 
\end{lemma}
\begin{proof}
    It follows from the energy estimates \eqref{E2}, \eqref{E3} and the triangle inequality that \begin{align}\begin{aligned}
       \|\sqrt{\rho_{n}}\partial_{x}p_{n}(\rho_{n})\|_{L^{\infty}_{t}L^{2}_{x}}^{2} 
 &=  \|\sqrt{\rho_{n}}(w_{n}-u_{n})\|_{L^{\infty}_{t}L^{2}_{x}}^{2}  \le     \left(  \|\sqrt{\rho_{n}}w_{n}\|_{L^{\infty}_{t}L^{2}_{x}}^{2} +  \|\sqrt{\rho_{n}}u_{n}\|_{L^{\infty}_{t}L^{2}_{x}}^{2} \right) \\[1ex] &\le  (T+3)E_{n}, \label{ub1}\end{aligned}
    \end{align} where $C>0$ denotes an arbitrary constant independent of $n$ and $T$. Using the definition of $p_{n}$, we infer that
    \begin{equation*}
       \partial_{x}H_{n} =  (\gamma_{n})^{-1}\rho_{n} \partial_{x} p_{n}(\rho_{n})
    \end{equation*} and so by virtue of Young's inequality and \eqref{ub1}, \begin{align} \notag
          \|\partial_{x}H_{n}\|_{L^{\infty}_{t}L^{1}_{x}} &\le (2\gamma_{n})^{-1} \esssup_{t \in [0,T]} \left( \int_{\mathbb{T}} \rho_{n}~dx +  \int_{\mathbb{T}} \rho_{n} (\partial_{x}p_{n}(\rho_{n}))^{2}~dx       \right) \\[1ex] \notag &=  (2\gamma_{n})^{-1} \left(\|\rho_{n}\|_{L^{\infty}_{t}L^{1}_{x}} + \|\sqrt{\rho_{n}}\partial_{x}p_{n}(\rho_{n})\|_{L^{\infty}_{t}L^{2}_{x}}^{2} \right) \\[1ex] &\le \frac{(T+4)E_{n}}{2\gamma_{n}} \le (T+4)E_{n} \label{ub1.1}
    \end{align} for $n$ sufficiently large such that $\gamma_{n} > 1/2$. We also know from Lemma \ref{bdentropy} that 
        $\|H_{n}\|_{L^{\infty}_{t}L^{1}_{x}} \le E_{n}$ and so one can deduce from \eqref{ub1.1} and the embedding $W^{1,1}_{x} \hookrightarrow L^{\infty}_{x}$ that
     \begin{equation*}
         \|H_{n}\|_{L^{\infty}_{t,x}} \le C(T+5)E_{n},
     \end{equation*} where $C>0$ arises due to the aforementioned embedding and is a constant independent of $n$ and $T$. Recalling the form of $H_{n}$ from \eqref{Hnform}, this implies
     \begin{equation*}
         \frac{1}{\gamma_{n}+1} \rho_{n}^{\gamma_{n}+1} \le C(T+5)E_{n},
     \end{equation*} from which the result follows.
\end{proof}
\subsection{A blow-up lemma} \label{sectionblowup}
We wish to prove the following result:
\begin{lemma}[Criteria for blow-up of regular solutions] \label{blowup}
    Suppose $(\rho_{n}, u_{n})$ is a regular solution to \eqref{a1} - \eqref{a2} on $[0,T^{*})$ with initial data $(\rho_{n}^{0}, u_{n}^{0}) \in H^{k}(\mathbb{T}) \times H^{k}(\mathbb{T})$ where $k \ge 2$. Then provided that 
    \begin{equation}
        \underline{\rho_{n}} := \inf_{t \in [0,T^{*})} \min_{x \in \mathbb{T}} \rho_{n}(t,x) > 0, 
    \end{equation}
    we have that
    \begin{equation} \label{blow-up-statement}
        \sup_{t \in [0,T^{*})} \|\rho_{n}\|_{L^{\infty}(0,t; H^{k})} + \sup_{t \in [0,T^{*})} \|u_{n}\|_{L^{\infty}(0,t; H^{k})} + \sup_{t \in [0,T^{*})} \|u_{n}\|_{L^{2}(0,t; H^{k+1})} < +\infty,
    \end{equation} and therefore the solution can be extended to a larger time interval $[0,T)$, where $T > T^{*}$. In other words, the solution does not lose regularity unless the density reaches $0$ somewhere in the domain.
\end{lemma}
\begin{remark}
    The extension of the solution past $T^{*}$ can be justified as follows. Thanks to \eqref{blow-up-statement}, the pair $(\rho_{n}, u_{n})$ admits a limit as $t \nearrow T^{*}$ and the left-sided derivatives satisfy the system at $t=T^{*}$. Since we have $\rho_{n}(T^{\star}, \cdot) > 0$, we can use the local existence result from Theorem \ref{localexistence} to obtain a unique solution on $[T^{*}, T^{*} + \epsilon)$  for some $\epsilon > 0$. The extension is then given by the concatenation of the solutions on $[0,T^{*}]$ and $[T^{*}, T^{*} + \epsilon)$.
\end{remark}
 Our proof is done by induction with respect to the regularity parameter $k$. The base step corresponds to showing that
\begin{equation*}
    \sup_{t \in [0,T^{*})} \|\rho_{n}\|_{L^{\infty}(0,t; H^{2})} + \sup_{t \in [0,T^{*})} \|u_{n}\|_{L^{\infty}(0,t; H^{2})} + \sup_{t \in [0,T^{*})} \|u_{n}\|_{L^{2}(0,t; H^{3})} < +\infty
\end{equation*} provided $(\rho_{n}^{0}, u_{n}^{0}) \in H^{2}(\mathbb{T}) \times H^{2}(\mathbb{T})$. This is the goal of the current subsection. We remark that this assumption on the initial data allows us to deduce from Theorem \ref{localexistence} (existence of local solutions) that our solution satisfies \begin{equation*}
          \rho_{n} \in C(0,T^{*}; H^{2}(\mathbb{T})), ~u_{n} \in C(0,T^{*}; H^{2}(\mathbb{T})) \cap L^{2}(0, T^{*} ; H^{3}(\mathbb{T))},
    \end{equation*} which allows us to justify the computations which will follow.  The inductive part of the proof is deferred to Appendix A.
\subsubsection{Analysis of the singular diffusion $V_{n}$}
We first define the function
\begin{equation*}
    V_{n} := \lambda_{n}(\rho_{n})\partial_{x}u_{n}
\end{equation*} and establish some basic properties in the form of the next two propositions. The function $V_{n}$ corresponds to the 'active potential' which was first introduced in \cite{Constantin_2020}. The authors of \cite{Constantin_2020} used the active potential to prove higher order regularity estimates for strong solutions to their system, which shares similarities with the system we are considering for $n$ fixed. The authors of \cite{HCL} also made use of this function to prove higher-order regularity estimates for strong solutions for the Aw-Rascle system with a singular pressure. We follow a similar procedure. Firstly, let us find the equation satisfied by $V_{n}$.
\begin{proposition}[Equation for the singular diffusion] \label{eqnV}
    Let $n \in \mathbb{N}^{+}$. Suppose that $(\rho_{n}, u_{n})$ is a regular solution with $k=2$ (in the sense of Definition \eqref{defstrongsoln}) to \eqref{a1}-\eqref{a2} on $\mathbb{T} \times [0,T]$ with initial data $(\rho_{n}^{0}, u_{n}^{0}) \in H^{2}(\mathbb{T}) \times H^{2}(\mathbb{T})$ and $\rho_{n} > 0$ on $[0,T]$. Then, $V_{n}:=\lambda_{n}(\rho_{n})\partial_{x}u_{n}$ satisfies the following equation almost everywhere in $\mathbb{T} \times [0,T]$:
\begin{align}
    \partial_{t} V_{n} + \left(u_{n} + \frac{\lambda_{n}(\rho_{n})}{\rho_{n}^{2}} \partial_{x}\rho_{n} \right) \partial_{x}V_{n} - \frac{\lambda_{n}(\rho_{n})}{\rho_{n}} \partial_{x}^{2}V_{n} = - \frac{(\lambda_{n}'(\rho_{n})\rho_{n} + \lambda_{n}(\rho_{n}))}{(\lambda_{n}(\rho_{n}))^{2}}V_{n}^{2}. \label{Veqn}
\end{align}
\end{proposition}
\begin{proof}
    For the sake of brevity we refer the reader to Lemma 3.7 of \cite{HCL} for a complete proof. 
\end{proof}
We now prove a first regularity estimate for $V_{n}$. In what follows $\epsilon_{i} > 0$ will be used to denote the constants arising from an application of Young's inequality where the index $i$ is used to distinguish between the different applications of Young's inequality. We also recall the Sobolev inequality
\begin{equation} \label{Sinf}
     \|u\|_{L^{\infty}(\mathbb{T})} \le \|u\|^{\frac{1}{2}}_{L^{2}(\mathbb{T})}\|\partial_{x}u\|^{\frac{1}{2}}_{L^{2}(\mathbb{T})} + \|u\|_{L^{2}(\mathbb{T})}, ~\forall ~u \in  H^{1}(\mathbb{T}). 
\end{equation} Under the assumptions of Proposition \ref{eqnV}, we have the following result.
\begin{proposition} \label{Vest1}
    $V_{n}$ satisfies \begin{align} \notag
        \|V_{n}\|^{2}_{L^{\infty}_{t}L^{2}_{x}}  &+ C_{3}   \left\|\partial_{x}V_{n}\right\|^{2}_{L^{2}_{t,x}} \\[1ex] &\le T\|V_{n}(0)\|_{L^{2}_{x}}^{2} \exp \left( 2C_{1}T + 2C_{2}\|V_{n}\|^{2}_{L^{2}_{t,x}} \right) \left(C_{1}T +C_{2}\|V_{n}(0)\|_{L^{2}_{x}}^{2} \right) =: \mathcal{V}_{1} \label{Vest1eqn}
    \end{align} and \begin{equation}
        \|V_{n}\|_{L^{2}_{t,x}} \le (\gamma_{n}(\overline{\rho_{n}})^{\gamma_{n}+1}E_{n})^{\frac{1}{2}},
    \end{equation}
    where $C_{3}$ is a positive constant depending on $n$ and
        \begin{align*}
    C_{1} \equiv C_{1}(\gamma_{n}, \underline{\rho_{n}}, \|R\|_{L^{\infty}_{t} L^{2}_{x}}, E_{n}), ~C_{2} \equiv C_{2}(\gamma_{n}, \underline{\rho_{n}}), ~C_{3} \equiv C(\gamma_{n}, \underline{\rho_{n}}).
\end{align*}
    \begin{proof}
       Multiplying \eqref{Veqn} by $V_{n}$ and integrating by parts leads to
        \begin{align*}
             \frac{1}{2}\frac{d}{dt} \int_{\mathbb{T}} |V_{n}|^{2}~dx + \int_{\mathbb{T}} \frac{\lambda_{n}(\rho_{n})}{\rho_{n}} (\partial_{x}V_{n})^{2}~dx = &- \int_{\mathbb{T}} V_{n} \partial_{x}V_{n} \frac{\partial_{x}\lambda_{n}(\rho_{n})}{\rho_{n}}~dx - \int_{\mathbb{T}} u_{n}V_{n}\partial_{x}V_{n}~dx \\[1ex] &- \int_{\mathbb{T}} \frac{\lambda_{n}'(\rho_{n})\rho_{n} + \lambda_{n}(\rho_{n})}{(\lambda_{n}(\rho_{n}))^{2}}V_{n}^{3}~dx =: \sum_{i=1}^{3}I_{1}.
        \end{align*} Defining $ \displaystyle R := \rho_{n}^{-1}\partial_{x}\lambda_{n}(\rho_{n})$, a direct computation reveals that $ R = \gamma_{n}(\gamma_{n}+1)\partial_{x}p_{n}(\rho_{n}),$ and so by \eqref{ub1}
        \begin{align*}
            \|R\|_{L^{\infty}_{t}L^{2}_{x}} \le \frac{\gamma_{n}(\gamma_{n}+1)}{\sqrt{\underline{\rho_{n}}}}(T+3)E_{n},
        \end{align*} where $E_{n} := E_{n}^{0} + E_{n}^{1} + E_{n}^{2}$. We now estimate $I_{1} - I_{3}$. Using the Holder and Sobolev inequalities,
        \begin{align*}
            I_{1} \le ~\|R\|_{L^{2}_{x}} \|\partial_{x}V_{n}\|_{L^{2}_{x}} \|V_{n}\|_{L^{\infty}_{x}} \le \|\partial_{x}V_{n}\|_{L^{2}_{x}} \left( \|V_{n}\|_{L^{2}_{x}}^{\frac{1}{2}}\|\partial_{x}V_{n}\|_{L^{2}_{x}}^{\frac{1}{2}} + \|V_{n}\|_{L^{2}_{x}} \right)\|R\|_{L^{2}_{x}} .
        \end{align*}
        Applying Young's inequality twice then gives us
         \begin{align*}
            I_{1} &\le \epsilon_{1}\|\partial_{x}V_{n}\|_{L^{2}_{x}}^{2} + \frac{1}{4\epsilon_{1}}\|R\|_{L^{2}_{x}}^{2} \left( \|V_{n}\|_{L^{2}_{x}}^{\frac{1}{2}}\|\partial_{x}V_{n}\|_{L^{2}_{x}}^{\frac{1}{2}} + \|V_{n}\|_{L^{2}_{x}}\right)^{2} \\ &\le (\epsilon_{1}+\epsilon_{2}) \|\partial_{x}V_{n}\|_{L^{2}_{x}}^{2} + \frac{1}{16\epsilon_{1}^{2}\epsilon_{2}}\|R\|_{L^{2}_{x}}^{4}\|V_{n}\|_{L^{2}_{x}}^{2} + \frac{1}{2\epsilon_{1}}\|V_{n}\|_{L^{2}_{x}}^{2}\|R\|_{L^{2}_{x}}^{2}.
        \end{align*} Next,
        \begin{align*}
            I_{2} \le \left| \int_{\mathbb{T}} u_{n}V_{n}\partial_{x}V_{n} \right| \le \|u_{n}\|_{L^{\infty}_{x}}\|V_{n}\|_{L^{2}_{x}}\|\partial_{x}V_{n}\|_{L^{2}_{x}} \le \epsilon_{3} \|\partial_{x}V_{n}\|_{L^{2}_{x}}^{2} + \frac{1}{4\epsilon_{3}} \|u_{n}\|^{2}_{L^{\infty}_{x}}\|V_{n}\|^{2}_{L^{2}_{x}}.
        \end{align*} Due to the Sobolev inequality \eqref{Sinf},
        \begin{align} \notag 
             \|u_{n}\|_{L^{\infty}_{x}}^{2} &\le \|u_{n}\|_{L^{2}_{x}}\|\partial_{x}u_{n}\|_{L^{2}_{x}} + \|u_{n}\|_{L^{2}_{x}}^{2} \le \frac{3}{2}\|u_{n}\|_{L^{2}_{x}}^{2} + \frac{1}{2}\|\partial_{x}u_{n}\|_{L^{2}_{x}}^{2} \\[1ex]
             &\le \frac{3}{2 \underline{\rho_{n}}}\|\sqrt{\rho_{n}}u_{n}\|_{L^{2}_{x}}^{2} + \frac{1}{2}\left\|\frac{1}{\lambda_{n}(\rho_{n})}\right\|^{2}_{L^{\infty}_{t,x}}\|V_{n}\|_{L^{2}_{x}}^{2} \\ &\le   \frac{3}{2 \underline{\rho_{n}}}\|\sqrt{\rho_{n}}u_{n}\|_{L^{2}_{x}}^{2} + \gamma_{n}^{-2} \underline{\rho_{n}}^{-2-2\gamma_{n}}\|V_{n}\|_{L^{2}_{x}}^{2}. 
        \end{align} Therefore,
    \begin{align*}
        I_{2} \le \epsilon_{3} \|\partial_{x}V_{n}\|_{L^{2}_{x}}^{2} + (\epsilon_{3}\underline{\rho_{n}})^{-1}E_{n}\|V_{n}\|_{L^{2}_{x}}^{2} + (4\gamma_{n}^{-1}\epsilon_{3}\underline{\rho_{n}}^{-1-\gamma_{n}})^{-1}\|V_{n}\|_{L^{2}_{x}}^{4}.
    \end{align*}
Moving onto $I_{3}$, we first remark that
\begin{align*}
    \frac{\lambda_{n}'(\rho_{n})\rho_{n}+ \lambda_{n}(\rho_{n})}{(\lambda_{n}(\rho_{n}))^{2}} = \frac{\gamma_{n}+2}{\gamma_{n}\rho_{n}^{\gamma_{n}+1}}.
\end{align*} Then using the Holder, Sobolev \eqref{Sinf} and Young's inequalities,
\begin{align*}
    I_{3} &\le \frac{\gamma_{n}+2}{\gamma_{n}\underline{\rho_{n}}^{\gamma_{n}+1}} \|V_{n}\|_{L^{\infty}_{x}}\|V_{n}\|_{L^{2}_{x}}^{2} \le \frac{\gamma_{n}+2}{\gamma_{n}\underline{\rho_{n}}^{\gamma_{n}+1}} \left( \|V_{n}\|_{L^{2}_{x}}^{\frac{5}{2}}\|\partial_{x}V_{n}\|_{L^{2}_{x}}^{\frac{1}{2}} + \|V_{n}\|_{L^{2}_{x}}^{3}  \right) \\ &\le \frac{\gamma_{n}+2}{\gamma_{n}\underline{\rho_{n}}^{\gamma_{n}+1}} \left( \epsilon_{4}^{-\frac{4}{3}}\|V_{n}\|_{L^{2}_{x}}^{\frac{10}{3}} + \epsilon_{4}^{4}\|\partial_{x}V_{n}\|_{L^{2}_{x}}^{2} + \|V_{n}\|_{L^{2}_{x}}^{3} \right). 
\end{align*} By two further applications of Young's inequality,
\begin{align*}
    I_{3} &\le \frac{\gamma_{n}+2}{\gamma_{n}\underline{\rho_{n}}^{\gamma_{n}+1}} \left(\epsilon_{4}^{-\frac{4}{3}} \|V_{n}\|_{L^{2}_{x}}^{2} + \epsilon_{4}^{4}\|\partial_{x}V_{n}\|_{L^{2}_{x}}^{2} + (\epsilon_{4})^{-\frac{16}{3}}\|V_{n}\|_{L^{2}_{x}}^{4} \right),
\end{align*} where we have used the following observation due to Young's inequality:
\begin{align*}
    &\|V_{n}\|^{\frac{10}{3}} = \|V_{n}\|^{2}\|V_{n}\|^{\frac{4}{3}} \le \frac{2}{3}\|V_{n}\|^{3} + \frac{1}{3}\|V_{n}\|^{4} \le \frac{1}{3}\|V_{n}\|^{2} + \frac{2}{3}\|V_{n}\|^{4}.
\end{align*} Assembling our estimates for $I_{1}- I_{3}$, we have
\begin{align*}
    \frac{1}{2}\frac{d}{dt}\|V_{n}\|_{L^{2}_{x}}^{2} + \int_{\mathbb{T}} \frac{\lambda_{n}(\rho_{n})}{\rho_{n}} (\partial_{x}V_{n})^{2}~dx &\le (\epsilon_{1} + \epsilon_{2} + \epsilon_{3} + \epsilon_{4}^{4} )\|\partial_{x}V_{n}\|_{L^{2}_{x}}^{2} + \frac{1}{16\epsilon_{1}^{2}\epsilon_{2}}\|R\|_{L^{2}_{x}}^{4}\|V_{n}\|_{L^{2}_{x}}^{2} + \frac{1}{2\epsilon_{1}}\|V_{n}\|_{L^{2}_{x}}^{2}\|R\|_{L^{2}_{x}}^{2} \\[1ex] &+ (\epsilon_{3}\underline{\rho_{n}})^{-1}E_{n}\|V_{n}\|_{L^{2}_{x}}^{2} + (4\gamma_{n}^{2}\epsilon_{3}\underline{\rho_{n}}^{-1-\gamma_{n}})^{-1}\|V_{n}\|_{L^{2}_{x}}^{4} \\[1ex] &+ \frac{\gamma_{n}+2}{\gamma_{n}\underline{\rho_{n}}^{\gamma_{n}+1}} \left( \|V_{n}\|_{L^{2}_{x}}^{2} + \epsilon_{4}^{4}\|\partial_{x}V_{n}\|_{L^{2}_{x}}^{2} + (\epsilon_{4})^{-\frac{16}{3}}\|V_{n}\|_{L^{2}_{x}}^{4} \right).
\end{align*}
Observing that $\rho_{n}^{-1} \lambda_{n}(\rho_{n}) = \gamma_{n}\rho_{n}^{\gamma_{n}} \ge \gamma_{n} \underline{\rho_{n}}^{\gamma_{n}}$ and simplifying the RHS of the above inequality, we have \begin{align*}
    \frac{1}{2}\frac{d}{dt}\|V_{n}\|_{L^{2}_{x}}^{2} &+ \left( \gamma_{n}\underline{\rho_{n}}^{\gamma_{n}} - \epsilon_{1} - \epsilon_{2} - \epsilon_{3} - \epsilon_{4}^{4} \right) \left\|\partial_{x}V_{n}\right\|_{L^{2}_{x}}^{2} \\[1ex] &\le \|V_{n}\|_{L^{2}_{x}}^{2} \left\{ \frac{\|R\|_{L^{\infty}_{t}L^{2}_{x}}^{4}}{16\epsilon_{1}^{2}\epsilon_{2}} + \frac{\|R\|_{L^{\infty}_{t}L^{2}_{x}}^{2}}{2\epsilon_{1}} + \frac{E_{n}}{\epsilon_{3}\underline{\rho_{n}}} + \frac{\gamma_{n}+2}{\gamma_{n}\underline{\rho_{n}}^{\gamma_{n}+1}} + \left[ \frac{(\gamma_{n}+2)\epsilon_{4}^{-\frac{16}{3}}}{\gamma_{n} \underline{\rho_{n}}^{\gamma_{n}+1}} + \frac{E_{n}}{4\gamma_{n}\epsilon_{3} \underline{\rho_{n}}^{\gamma_{n}-1}} \right]\|V_{n}\|_{L^{2}_{x}}^{2}            \right\}.
\end{align*}
Choosing $\epsilon_{i} > 0$ small enough (possibly dependent on $n$), we have
\begin{align}
\begin{aligned}
     \frac{1}{2}\frac{d}{dt}\|V_{n}\|_{L^{2}_{x}}^{2} &+ C_{3}   \left\|\partial_{x}V_{n}\right\|_{L^{2}_{x}}^{2} \le \|V_{n}\|_{L^{2}_{x}}^{2} \left\{ C_{1} + C_{2}\|V_{n}\|_{L^{2}_{x}}^{2}            \right\},   \label{3.3a}
     \end{aligned}
\end{align}where 
\begin{align*}
    &C_{1} := \frac{\|R\|_{L^{\infty}_{t}L^{2}_{x}}^{4}}{16\epsilon_{1}^{2}\epsilon_{2}} + \frac{\|R\|_{L^{\infty}_{t}L^{2}_{x}}^{2}}{2\epsilon_{1}} + \frac{E_{n}}{\epsilon_{3}\underline{\rho_{n}}} + \frac{\gamma_{n}+2}{\gamma_{n}\underline{\rho_{n}}^{\gamma_{n}+1}}, ~
    C_{2} := \frac{(\gamma_{n}+2)\epsilon_{4}^{-\frac{16}{3}}}{\gamma_{n} \underline{\rho_{n}}^{\gamma_{n}+1}} + \frac{E_{n}}{4\gamma_{n}\epsilon_{3} \underline{\rho_{n}}^{\gamma_{n}-1}}, ~C_{3} := \frac{1}{5}\gamma_{n}\underline{\rho_{n}}^{\gamma_{n}}.
\end{align*}
An application of Gronwall's inequality yields
\begin{align*}
    \frac{1}{2}\|V_{n}(t)\|^{2}_{L^{\infty}_{t}L^{2}_{x}} \le \|V_{n}(0)\|_{L^{2}_{x}}^{2} \exp \left( C_{1}T + C_{2}\|V_{n}\|^{2}_{L^{2}_{t,x}} \right).
\end{align*} 
Exploiting this inequality in \eqref{3.3a} leads to
\begin{align*}
     \frac{1}{2}\frac{d}{dt}\|V_{n}\|_{L^{2}_{x}}^{2} &+ C_{3}   \left\|\partial_{x}V_{n}\right\|_{L^{2}_{x}}^{2} \\[1ex] &\le \|V_{n}(0)\|_{L^{2}_{x}}^{2} \exp \left( C_{1}T + C_{2}\|V_{n}\|^{2}_{L^{2}_{t,x}} \right) \left\{C_{1}T +C_{2}\|V_{n}(0)\|_{L^{2}_{x}}^{2} \exp \left( C_{1}T + C_{2}\|V_{n}\|^{2}_{L^{2}_{t,x}} \right)\right\}.
\end{align*} Integrating in time and simplifying leads to \eqref{Vest1eqn}.
To finish, we notice that due to \eqref{E2},  \begin{equation*}
   \|V_{n}\|_{L^{2}_{t,x}}^{2} =  \|\lambda_{n}(\rho_{n})^{2}(\partial_{x}u_{n})^{2}\|_{L^{1}_{t,x}} \le \gamma_{n}\overline{\rho_{n}}^{\gamma_{n}+1}\|\lambda_{n}(\rho_{n})(\partial_{x}u_{n})^{2}\|_{L^{1}_{t,x}} \le \gamma_{n}\overline{\rho_{n}}^{\gamma_{n}+1}E_{n}.
\end{equation*}
 \end{proof}
\end{proposition}
The previous result allows us to deduce some regularity estimates on $u_{n}$ and $\rho_{n}$. In the remainder of this section we denote by $M_{k}^{n}$ an arbitrary function which satisfies
\begin{equation*}
    M_{k}^{n} \equiv M_{k}^{n}(t, \|\rho_{n}\|_{L^{\infty}_{t}H^{k}_{x}}, \|u_{n}\|_{L^{\infty}_{t}H^{k}_{x}}, \|u_{n}\|_{L^{2}_{t}H^{k+1}_{x}}, \gamma_{n}, E_{n},  \overline{\rho_{n}}, (\underline{\rho_{n}})^{-1}), ~~ \sup_{t \in [0,T^{*})} M_{k}^{n} < + \infty.
\end{equation*}
\begin{corollary} \label{Vest2}
    We have that
    \begin{enumerate}
        \item $\partial_{x}u_{n}$ is bounded in $L^{2}_{t}L^{2}_{x}$ with \begin{equation}
            \|\partial_{x}u_{n}\|^{2}_{L^{2}_{t}L^{2}_{x}} \le \frac{E_{n}}{\gamma_{n}\underline{\rho_{n}}^{\gamma_{n}+1}}.    \end{equation} 
        \item $\partial_{x}u_{n}$ is bounded in $L^{\infty}_{t}L^{2}_{x}$ with \begin{align}
             \|\partial_{x}u_{n}\|_{L^{\infty}_{t}L^{2}_{x}} \le \frac{\mathcal{V}_{1}}{\gamma_{n}\underline{\rho_{n}}^{\gamma_{n}+1}} \le M_{0}^{n} .   \end{align}
         \item $\partial_{x} \rho_{n}$ is bounded in $L^{\infty}_{t}L^{2}_{x}$ with \begin{align}
            \|\partial_{x}\rho_{n}\|_{L^{\infty}_{t}L^{2}_{x}} \le \gamma_{n}^{-1}\underline{\rho_{n}}^{\frac{1}{2}-\gamma_{n}}E_{n} \le M_{0}^{n}.
        \end{align}
        \item $\partial_{x}^{2}u_{n}$ is bounded in $L^{2}_{t}L^{2}_{x}$ with \begin{align}\begin{aligned}
            \|\partial_{x}^{2}u_{n}\|_{L^{2}_{t}L^{2}_{x}} &\le \frac{1}{\gamma_{n} \underline{\rho_{n}}^{\gamma_{n}+1}}\left( \|\partial_{x}V_{n}\|_{L^{2}_{t}L^{2}_{x}} + \|\partial_{x}\lambda_{n}(\rho_{n})\|_{L^{\infty}_{t}L^{2}_{x}}\|\partial_{x}u_{n}\|_{L^{2}_{t}L^{2}_{x}}^{\frac{1}{2}} \left(\|\partial_{x}^{2}u_{n}\|_{L^{2}_{t}L^{2}_{x}}^{\frac{1}{2}} + \|\partial_{x}V_{n}\|_{L^{2}_{t}L^{2}_{x}} \right) \right) \\[1ex] &\le M_{0}^{n}
        \end{aligned}
        \end{align}
    \end{enumerate}
\end{corollary}
\begin{proof}
    The first bound follows from \eqref{E2}. For the second bound, we use the relationship $V_{n} = \lambda_{n}(\rho_{n})\partial_{x}u_{n}$ to estimate
    \begin{align*}
        \|\partial_{x}u_{n}\|_{L^{\infty}_{t}L^{2}_{x}} \le \left\| \frac{1}{\lambda_{n}(\rho_{n})} \right\|_{L^{\infty}_{t,x}}\|V_{n}\|_{L^{\infty}_{t}L^{2}_{x}} \le \frac{1}{\gamma_{n} \underline{\rho_{n}}^{\gamma_{n}+1}}\|V_{n}\|_{L^{\infty}_{t} L^{2}_{x}} \le \frac{\mathcal{V}_{1}}{\gamma_{n} \underline{\rho_{n}}^{\gamma_{n}+1}}.
    \end{align*} 
    For the third estimate, we first note that from \eqref{E2} and \eqref{E3}, we have \begin{align} \begin{aligned}
        \|\sqrt{\rho_{n}} \partial_{x}p_{n}(\rho_{n})\|_{L^{\infty}_{t}L^{2}_{x}} &= \|\sqrt{\rho_{n}}(w_{n}-u_{n}) \|_{L^{\infty}_{t}L^{2}_{x}} \le  \|\sqrt{\rho_{n}} w_{n}\|_{L^{\infty}_{t}L^{2}_{x}} + \|\sqrt{\rho_{n}}u_{n}\|_{L^{\infty}_{t}L^{2}_{x}} \\[1ex] &\le \|\sqrt{\rho_{n}^{0}}u_{n}^{0}\|^{2}_{L^{2}_{x}} + \|\sqrt{\rho_{n}^{0}}w_{n}^{0}\|^{2}_{L^{2}_{x}} \le E_{n}. \label{3.3b}\end{aligned}
    \end{align}
    On the other hand, \begin{equation}
    \begin{aligned}
          \|\sqrt{\rho_{n}} \partial_{x}p_{n}(\rho_{n})\|_{L^{\infty}_{t}L^{2}_{x}} = \| \gamma_{n} \rho_{n}^{\gamma_{n}-\frac{1}{2}}\partial_{x} \rho_{n}\|_{L^{\infty}_{t}L^{2}_{x}} \ge \gamma_{n} \underline{\rho_{n}}^{\gamma_{n}-\frac{1}{2}}\|\partial_{x}\rho_{n}\|_{L^{\infty}_{t}L^{2}_{x}}. \label{3.3c}
        \end{aligned}
    \end{equation} The result follows directly from \eqref{3.3b} and \eqref{3.3c}. In order to deduce the fourth bound, we differentiate $V_{n} = \lambda_{n}(\rho_{n})\partial_{x}u_{n}$ to find
    \begin{equation*}
        \partial_{x}^{2}u_{n} = \frac{\partial_{x}V_{n} - \partial_{x}\lambda_{n}(\rho_{n})\partial_{x}u_{n}}{\lambda_{n}(\rho_{n})}.
    \end{equation*} Thus, we have
    \begin{align*}
        \|\partial_{x}^{2}u_{n}\|_{L^{2}_{t}L^{2}_{x}} \le \left\| \frac{1}{\lambda_{n}(\rho_{n})} \right\|_{L^{\infty}_{t,x}} \left( \|\partial_{x}V_{n}\|_{L^{2}_{t}L^{2}_{x}} + \|\partial_{x} \lambda_{n} \partial_{x}u_{n}\|_{L^{2}_{t}L^{2}_{x}} \right).
    \end{align*} Noticing that $\|\partial_{x} \lambda_{n} \partial_{x}u_{n}\|_{L^{2}_{t}L^{2}_{x}} \le \|\partial_{x}\lambda_{n}(\rho_{n})\|_{L^{\infty}_{t}L^{2}_{x}}\|\partial_{x}u_{n}\|_{L^{2}_{t}L^{\infty}_{x}}$ allows us to further deduce using the Sobolev inequality \eqref{Sinf} that
    \begin{align*}
         \|\partial_{x}^{2}u_{n}\|_{L^{2}_{t}L^{2}_{x}} \le \frac{1}{\gamma_{n} \underline{\rho_{n}}^{\gamma_{n}+1}} \left( \|\partial_{x}V_{n}\|_{L^{2}_{t}L^{2}_{x}} + \|\partial_{x}\lambda_{n}(\rho_{n})\|_{L^{\infty}_{t}L^{2}_{x}} (\|\partial_{x}u_{n}\|_{L^{2}_{t}L^{2}_{x}}^{\frac{1}{2}} + \|\partial_{x}^{2}u_{n}\|_{L^{2}_{t}L^{2}_{x}}^{\frac{1}{2}} + \|\partial_{x}u_{n}\|_{L^{2}_{t}L^{2}_{x}})    \right).
    \end{align*} A Gronwall argument can then be used to bound this quantity by $M_{0}^{n}$. 
\end{proof}
\begin{remark}
   It follows from the energy estimates \eqref{E1}, \eqref{E2} and \eqref{E3} and the upper/lower bounds on the density that $\|\rho_{n}\|_{L^{\infty}_{t}L^{2}_{x}} + \|u_{n}\|_{L^{\infty}_{t}L^{2}_{x}}  + \|u_{n}\|_{L^{2}_{t}H^{1}_{x}} \le C$, where $C$ is independent of time. Combining this with Corollary \ref{Vest2}, the estimates we have acquired up to this point can be summarised as:
    \begin{equation}
        \|\rho_{n}\|_{L^{\infty}_{t}H^{1}_{x}} + \|u_{n}\|_{L^{\infty}_{t}H^{1}_{x}}  + \|u_{n}\|_{L^{2}_{t}H^{2}_{x}} \le M_{0}^{n}. \label{buk1}
    \end{equation} 
\end{remark}
\subsubsection{Completing the base step of induction}
To complete the base step we need to show that \begin{equation} \label{bugoalk2}
\|\partial_{x}^{2}\rho_{n}\|_{L^{\infty}_{t}L^{2}_{x}} + \|\partial_{x}^{2}u_{n}\|_{L^{\infty}_{t}L^{2}_{x}} + \|\partial_{x}^{3}u_{n}\|_{L^{2}_{t,x}} < M_{1}^{n}. 
\end{equation}
We now make note of a result which will allow us to deduce the remaining higher-order regularity estimates. We assume $\rho_{n}, u_{n}$ are smooth for the sake of convenience.
\begin{proposition}(Higher order regularity formulae)
    Let $k \ge 2$, ~$\rho_{n} \in C^{1}(0,T; C^{k}(\mathbb{T})), ~u_{n} \in C(0,T; C^{k+1}(\mathbb{T}))$ satisfy \eqref{a1}. Then,
    \begin{align}
        \frac{1}{2}\frac{d}{dt} \|\partial_{x}^{k} \rho_{n}\|_{L^{2}_{x}}^{2} \le C\left( \|\partial_{x}^{k}u_{n}\|_{L^{2}_{x}} \|\partial_{x}^{k}\rho_{n}\|_{L^{2}_{x}}^{2} + \|\rho_{n}\|_{L^{\infty}_{x}} \|\partial_{x}^{k}\rho_{n}\|_{L^{2}_{x}}\|\partial_{x}^{k+1}u_{n}\|_{L^{2}_{x}}\right) \label{HOR1}
    \end{align}
    Additionally, for $l \ge 1$ and $\rho_{n} \in C(0,T; C^{l+1}(\mathbb{T})),~u_{n} \in C(0,T; C^{l}(\mathbb{T})),~V_{n} \in C^{1}(0,T; C^{l+1}(\mathbb{T}))$ satisfying \eqref{Veqn}, we have
    \begin{align} \begin{aligned}
     \frac{1}{2}\frac{d}{dt}\int_{\mathbb{T}} |\partial_{x}^{l}V_{n}|^{2}~dx &= - \int_{\mathbb{T}} \partial_{x}^{l} \left( (u_{n} + \frac{\lambda_{n}(\rho_{n})}{\rho_{n}^{2}}\partial_{x}\rho_{n}) \partial_{x}V_{n}\right) \partial_{x}^{l}V_{n}~dx + \int_{\mathbb{T}} \partial_{x}^{l} \left( \frac{\lambda_{n}(\rho_{n})}{\rho_{n}}\partial_{x}^{2}V_{n}\right)\partial_{x}^{l}V_{n} \\[1ex]  &- \int_{\mathbb{T}} \partial_{x}^{l} \left( \frac{\lambda_{n}'(\rho_{n})\rho_{n} + \lambda_{n}(\rho_{n})}{(\lambda_{n}(\rho_{n}))^{2}}V_{n}^{2}\right)\partial_{x}^{l}V_{n}~dx. \label{HOR2} \end{aligned} 
    \end{align}
\end{proposition}
\begin{proof}
    The proof is identical to that of Lemma 3.10 in \cite{HCL}.
\end{proof}
Taking $l = 1$ in \eqref{HOR2}, one can integrate by parts to obtain
\begin{align*}
    \frac{1}{2}\frac{d}{dt} \int_{\mathbb{T}} |\partial_{x}V_{n}|^{2}~dx &+ \int_{\mathbb{T}} \frac{\lambda_{n}(\rho_{n})}{\rho_{n}} |\partial_{x}^{2}V_{n}|^{2}~dx \\[1ex] &= \int_{\mathbb{T}} \left(u_{n} + \frac{\lambda_{n}(\rho_{n})}{\rho_{n}^{2}} \partial_{x}\rho_{n}\right) \partial_{x}V_{n} \partial_{x}^{2}V_{n}~dx + \int_{\mathbb{T}} \frac{\lambda_{n}'(\rho_{n})\rho_{n} + \lambda_{n}(\rho_{n})}{(\lambda_{n}(\rho_{n}))^{2}}V_{n}^{2} \partial_{x}^{2}V_{n}~dx =: J_{1} + J_{2}.
\end{align*}
Due to \eqref{buk1}, we have
\begin{align*}
    J_{1} &\le \|u_{n} + \rho_{n}^{-2}\lambda_{n}(\rho_{n})\partial_{x}\rho_{n}\|_{L^{2}_{x}}\|\partial_{x}V_{n}\|_{L^{\infty}_{x}}\|\partial_{x}^{2}V_{n}\|_{L^{2}_{x}} \\[1ex] &\le M_{1}^{n}\|\partial_{x}V_{n}\|_{L^{\infty}_{x}}\|\partial_{x}^{2}V_{n}\|_{L^{2}_{x}}.
\end{align*} Using the Sobolev inequality \eqref{Sinf} and Young's inequality, we arrive at \begin{align*}
    J_{1} \le  \epsilon \|\partial_{x}^{2}V_{n}\|_{L^{2}_{x}}^{2} + M_{1}^{n} \|\partial_{x}V_{n}\|_{L^{2}_{x}}^{2} \le \epsilon \|\partial_{x}^{2}V_{n}\|_{L^{2}_{x}}^{2} + M_{1}^{n}.
\end{align*} Meanwhile,
\begin{align*}
    J_{2} \le M_{1}^{n} \|V_{n}\|_{L^{4}_{x}}^{2}\|\partial_{x}^{2
    }V_{n}\|_{L^{2}_{x}} \le M_{1}^{n} + \epsilon\|\partial_{x}^{2}V_{n}\|_{L^{2}_{x}}^{2},
\end{align*} using Young's inequality and Proposition \ref{Vest1}. Therefore we have
\begin{align*}
     \frac{1}{2}\frac{d}{dt} \int_{\mathbb{T}} |\partial_{x}V_{n}|^{2}~dx + \int_{\mathbb{T}} \frac{\lambda_{n}(\rho_{n})}{\rho_{n}} |\partial_{x}^{2}V_{n}|^{2}~dx \le 2\epsilon\|\partial_{x}^{2}V_{n}\|_{L^{2}_{x}}^{2} + M_{1}^{n}
\end{align*} Using the observation $\rho_{n}^{-1}\lambda_{n}(\rho_{n}) = \gamma_{n}\rho_{n}^{\gamma_{n}} \ge \gamma_{n} \underline{\rho_{n}}^{\gamma_{n}}$ and choosing $\epsilon$ small enough (depending on $n$), we have
\begin{align} \label{Vest3}
     \frac{1}{2}\frac{d}{dt} \int_{\mathbb{T}} |\partial_{x}V_{n}|^{2}~dx + C  \|\partial_{x}^{2}V_{n}\|^{2}_{L^{2}_{x}} \le  M_{1}^{n},
\end{align} where $C > 0$ is a constant depending on $n$. Integrating in time gives us
\begin{equation}
    \|\partial_{x}V_{n}\|_{L^{\infty}_{t}L^{2}_{x}} + \|\partial_{x}^{2}V_{n}\|_{L^{2}_{t,x}} \le M_{1}^{n}.
\end{equation} Next, from the relationship $V_{n} = \lambda_{n}\partial_{x}u_{n}$, we infer that
\begin{equation} \label{dxku2}
    \partial_{x}^{2}u_{n} = \lambda_{n}^{-1}\partial_{x}V_{n} + V_{n} \partial_{x} (\lambda_{n}^{-1})
\end{equation} and so
\begin{align} \label{dxku2.1}
    \|\partial_{x}^{2}u_{n}\|_{L^{\infty}_{t}L^{2}_{x}} &\le \|\lambda_{n}^{-1}\|_{L^{\infty}_{t,x}}\|\partial_{x}V_{n}\|_{L^{\infty}_{t}L^{2}_{x}} + \|V_{n}\|_{L^{\infty}_{t,x}}\|\partial_{x} (\lambda_{n}^{-1})\|_{L^{\infty}_{t}L^{2}_{x}} \le M_{1}^{n}.
\end{align}
Differentiating \eqref{dxku2} once more,
\begin{equation*}
    \partial_{x}^{3}u_{n} = \partial_{x}^{2}V_{n} \lambda_{n}^{-1} + 2\partial_{x}V_{n}\partial_{x} \lambda_{n}^{-1} + V_{n} \partial_{x}^{2} \lambda_{n}^{-1}
\end{equation*} and again we estimate
\begin{equation} \label{dx3u1}
    \|\partial_{x}^{3}u_{n}\|_{L^{2}_{t,x}} \le \|\partial_{x}^{2}V_{n}\|_{L^{2}_{t,x}}\|\lambda_{n}^{-1}\|_{L^{\infty}_{t,x}} + 2\|\partial_{x}V_{n}\|_{L^{\infty}_{t}L^{2}_{x}}\|\partial_{x} \lambda_{n}^{-1}\|_{L^{2}_{t}L^{\infty}_{x}} + \|V_{n}\|_{L^{\infty}_{t,x}}\|\partial_{x}^{2} \lambda_{n}^{-1}\|_{L^{2}_{t,x}}.
\end{equation} Note that $\partial_{x} \lambda_{n}^{-1} = -(1+\gamma_{n}^{-1})\rho_{n}^{-\gamma_{n}-2}\partial_{x}\rho_{n}$ and so $\|\partial_{x} \lambda_{n}^{-1}\|_{L^{2}_{t}L^{\infty}_{x}} \le M_{1}^{n} \|\partial_{x}\rho_{n}\|_{L^{2}_{t}L^{\infty}_{x}}$. Using the Sobolev inequality \eqref{Sinf} we can show that
\begin{equation*}
    \|\partial_{x}\rho_{n}\|_{L^{2}_{t}L^{\infty}_{x}} \le M_{1}^{n} + M_{1}^{n}\|\partial_{x}^{2}\rho_{n}\|_{L^{2}_{t,x}}.
\end{equation*}
By computing $\partial_{x}^{2}\lambda_{n}^{-1}$ explicitly one can also show that $\|\partial_{x}^{2}\lambda_{n}^{-1}\|_{L^{2}_{t,x}} \le M_{1}^{n} + M_{1}^{n} \|\partial_{x}^{2}\rho_{n}\|_{L^{2}_{t,x}}$. Returning to \eqref{dx3u1}, we have
\begin{equation} \label{dx3u2}
     \|\partial_{x}^{3}u_{n}\|_{L^{2}_{t,x}} \le M_{1}^{n} + M_{1}^{n}\|\partial_{x}^{2}\rho_{n}\|_{L^{2}_{t,x}}.
\end{equation} Taking $k=2$ in \eqref{HOR1}, we have
\begin{align*}
    \frac{1}{2}\frac{d}{dt}\|\partial_{x}^{2}\rho_{n}\|_{L^{2}_{x}}^{2} &\le C\left( \|\partial_{x}^{2}u_{n}\|_{L^{2}_{x}} \|\partial_{x}^{2}\rho_{n}\|^{2}_{L^{2}_{x}} + \|\rho_{n}\|_{L^{\infty}_{x}} \|\partial_{x}^{2}\rho_{n}\|_{L^{2}_{x}}\|\partial_{x}^{3}u_{n}\|_{L^{2}_{x}}\right) \\[1ex] &\le M_{1}^{n}\|\partial_{x}^{2}\rho_{n}\|_{L^{2}_{x}},
\end{align*} using \eqref{dx3u2} and the bounds previously obtained. An application of Gronwall's inequality grants us $\|\partial_{x}^{2}\rho_{n}\|_{L^{\infty}_{t}L^{2}_{x}} \le M_{1}^{n}$. In light of this, the inequality \eqref{dx3u2} finally gives us $\|\partial_{x}^{3}u_{n}\|_{L^{2}_{t,x}} \le M_{1}^{n}$. Additionally recalling the bound \eqref{dxku2.1}, we have shown that \eqref{bugoalk2} holds true. Thus we have 
\begin{equation}
        \|\rho_{n}\|_{L^{\infty}_{t}H^{2}_{x}} + \|u_{n}\|_{L^{\infty}_{t}H^{2}_{x}}  + \|u_{n}\|_{L^{2}_{t}H^{3}_{x}} \le M_{1}^{n}, \label{buk2}
\end{equation} 
where $M_{1}^{n}$ satisfies $\displaystyle \sup_{t \in [0,T^{*})}M_{1}^{n} < + \infty$. Thus the base step is complete. The inductive step is deferred to the Appendix. With this, we conclude the proof of Lemma \ref{blowup}.

\subsection{Completing the proof of global existence} \label{sectionproofglobal}
We are now ready to complete the proof of Theorem \ref{globalexistence} and obtain the existence of a global unique regular solution.
\begin{proof}[Proof of Theorem \ref{globalexistence}]
We have previously asserted the existence of a unique global regular solution $(\rho_{n}, u_{n})$ to \eqref{a1}-\eqref{a2} on $[0,T_{0}]$, for some $T_{0} > 0$ in Theorem \ref{localexistence}. The blow-up result we obtained in the previous subsection (Lemma \ref{blowup}) also tells us that our solution can in fact be extended past $T_{0}$ provided the density $\rho_{n}$ does not hit $0$ anywhere in $\mathbb{T} \times [0,T_{0}]$. Since we know that $\rho_{n} \ge \frac{r_{0}}{2} > 0$ on $\mathbb{T} \times [0,T_{0}]$, we can extend this solution and declare that $(\rho_{n}, u_{n})$ will possess a maximal interval of existence $[0, T^{*}),$ where $T^{*} > T_{0}$. For the sake of a contradiction, let us assume that $T^{*}$ is finite. Two particular implications of this statement are that \begin{equation} \label{3.4a}
\rho_{n}(\cdot,t) > 0 \text{ for each }t \in [0,T^{*}),\end{equation} and \begin{equation} \label{3.4b}
    \inf_{t \in [0,T^{*})} \min_{x \in \mathbb{T}} \rho_{n}(t,x) = 0.
\end{equation}
Let us now recall the original form of the Aw-Rascle system on our maximal interval of existence
\begin{subequations} 
\newcommand{\mystrut}{\vphantom{\pder{}{}}}
\begin{numcases}{}
    \partial_{t} \rho_{n} + \partial_{x}(\rho_{n} u_{n}) = 0, &\text{ in } $\mathbb{T} \times [0, T^{*}), $ \label{a}  \\[1ex]
    \partial_{t} (\rho_{n} w_{n}) + \partial_{x} (\rho_{n} w_{n} u_{n}) = 0, &\text{ in } $\mathbb{T} \times [0, T^{*}). $ \label{b} 
\end{numcases}
\end{subequations}
Since we know that $(\rho_{n}, u_{n})$ possesses the same regularity as given in the statement of this theorem, we can proceed in a similar way to \cite{Burtea_2020} and \cite{Constantin_2020} and define
\begin{equation} \label{W}
    W_{n} = \frac{\partial_{x}w_{n}}{\rho_{n}}.
\end{equation} We wish to show that $W_{n}$ is bounded from above uniformly in time. It will be seen that this is sufficient to obtain a lower bound on $\rho_{n}$, which will become clear when we derive the evolution equation for $1/\rho_{n}$ and apply a maximum-principle argument. Our $W_{n}$ corresponds to the so-called ‘effective pressure' used by Burtea and Haspot in \cite{Burtea_2020} and the function ‘$X$' by Constantin et al in Theorem 1.5, \cite{Constantin_2020}.  A straightforward computation reveals that
\begin{equation} \label{3.4c}
    \rho_{n}^{2}\partial_{t}W_{n} = \rho_{n} \partial_{t} \partial_{x} w_{n} - \partial_{x}w_{n} \partial_{t} \rho_{n}.
\end{equation}
Exploiting \eqref{a}, dividing by $\rho_{n} > 0$ and differentiating in space, one can show that \eqref{b} is equivalent to
\begin{equation} \label{3.4e}
    \partial_{t}\partial_{x}w_{n} = - \partial_{x} u_{n} \partial_{x} w_{n} - u_{n} \partial_{x}^{2}w_{n}.
\end{equation} Substituting \eqref{3.4e} and \eqref{a} into \eqref{3.4c} and dividing by $\rho_{n} > 0$ gives
\begin{align} \notag
   \rho_{n} \partial_{t}W_{n} &= -\rho_{n}W_{n}\partial_{x}u_{n} - u_{n} \partial_{x}^{2}w_{n} - \frac{\partial_{x}w_{n}\partial_{t}\rho_{n}}{\rho_{n}} \\[1ex] &=  - \rho_{n} \partial_{x} w_{n} W_{n} - u_{n} \partial_{x}\rho_{n} W_{n} - u_{n} \rho_{n} \partial_{x}W_{n} + W_{n}(\rho_{n}\partial_{x}u_{n} + u_{n} \partial_{x} \rho_{n})     \notag       \\[1ex]
    &= - u_{n}\rho_{n} \partial_{x}W_{n}. \label{wntransport}
\end{align} where we have used the facts that $\partial_{x}w_{n} = W_{n}\rho_{n}$ and $\partial_{x}^{2}w_{n} = W_{n}\partial_{x}\rho_{n} + \rho_{n}\partial_{x}W_{n}$, which both follow from \eqref{W}. Dividing by $\rho_{n}$ leaves us with
\begin{equation} \label{3.4f}
    \partial_{t}W_{n} + u_{n}\partial_{x}W_{n}=0.
\end{equation} It is clear that $w_{n}$ and $W_{n}$ satisfy the same transport equation. Now, we note that $w_{n} \in C(0,T; H^{3})$ due to the relation $w_{n} = u_{n} + \partial_{x}p_{n}(\rho_{n})$ and the regularity $\rho_{n}, u_{n} \in C(0,T; H^{4})$ for all $T < T^{*}$. This allows us to deduce that $W_{n} \in C(0,T; H^{2}) \subset C([0,T] \times \mathbb{T})$ for all $T < T^{*}$. The equation \eqref{3.4f} then implies that $\partial_{t}W_{n} \in C(0,T; H^{1}) \subset C( [0,T] \times \mathbb{T})$ and so $W_{n} \in C^{1}([0,T] \times \mathbb{T})$ for all $T < T^{*}$. We now introduce a maximum function $W_{n}^{M} : \mathbb{R}^{+} \to \mathbb{R}$ associated to $W_{n}$.
For each $t \in [0,T^{*})$, it is true that $W_{n}(\cdot, t)$ attains a maximum at some point $x_{t} \in \mathbb{T}$, thanks to the regularity of $W_{n}$ on $\mathbb{T} \times [0,T^{*})$. Thus, the map $t \to x_{t}$ is well-defined on $[0,T^{*})$ and the function
\begin{equation}
    W_{n}^{M}(t) := W_{n}(x_{t}, t) \equiv \max_{x \in \mathbb{T}} W_{n}(x,t)
\end{equation} is also well-defined and Lipschitz continuous. The Rademacher theorem implies that $W_{n}^{M}$ is differentiable almost everywhere on $[0,T^{*})$. We now wish to prove that \begin{equation} \label{3.4g}
    \left(W_{n}^{M} \right)'(t) = \partial_{t} W_{n}(x_{t},t)
\end{equation} for all $t \in [0,T^{*})$ where $W_{n}^{M}$ is differentiable. Using the classical definition of the derivative, we have for all $t$ where $W_{n}^{M}$ is differentiable,
\begin{align*}
\left(W_{n}^{M}\right)'(t) &= \lim_{h \searrow 0}   \left[ \frac{W_{n}^{M}(t+h) - W_{n}^{M}(t)}{h} \right] = \lim_{h \searrow 0} \left[ \frac{W_{n}(x_{t+h},t+h) - W_{n}(x_{t},t))}{h} \right]  \\[1ex] &\ge  \lim_{h \searrow 0} \left[ \frac{W_{n}(x_{t},t+h) - W_{n}(x_{t},t)}{h}  \right] = \partial_{t}W_{n}(x_{t},t). 
\end{align*} On the other hand, using an alternative definition of the derivative, we have
\begin{align*}
    \left(W_{n}^{M} \right)'(t) &= \lim_{h \searrow 0} \left[ \frac{W_{n}^{M}(t) - W_{n}^{M}(t-h)}{h}   \right] = \lim_{h \searrow 0} \left[ \frac{W_{n}(x_{t},t) - W_{n}(x_{t-h}, t-h)}{h} \right] \\[1ex] &\le \lim_{h \searrow 0} \left[ \frac{W_{n}(x_{t},t)-W_{n}(x_{t},t-h)}{h} \right] = \partial_{t} W_{n}(x_{t},t),
\end{align*} and so \eqref{3.4g} holds true. Evaluating \eqref{3.4f} along the points $(x_{t},t)$ which correspond to the maximum points of $W_{n}$, we have
\begin{equation}
    \partial_{t}W_{n}^{M}(t) = - u_{n}(x_{t},t)\underbrace{\partial_{x}W_{n}(x_{t},t)}_{=0} = 0,
\end{equation} for all $t \in [0,T^{*})$ where $W_{n}^{M}$ is differentiable, since $x_{t}$ is a maximum point of $W_{n}(\cdot, t)$. This tells us that 
\begin{equation*}
    \max_{x \in \mathbb{T}} \frac{\partial_{x}w_{n}}{\rho_{n}}(x,t) = \max_{x \in \mathbb{T}} \frac{\partial_{x}w_{n}^{0}}{\rho_{n}^{0}}(x)
\end{equation*} for any $t \in [0,T^{*})$, and as a consequence
\begin{equation} \label{dxwbound}
    \left( \frac{\partial_{x}w_{n}}{\rho_{n}
} \right)(x,t) \le \max_{\mathbb{T}} \frac{\partial_{x}w_{n}^{0}}{\rho_{n}^{0}} =: \mathcal{M}_{n}^{0}, ~\forall ~(x,t) \in \mathbb{T} \times [0,T^{*}).
\end{equation} We note that since we are working with periodic boundary conditions, the quantity $\max_{\mathbb{T}} \partial_{x}w_{n}^{0}$ is non-negative and therefore so is $\mathcal{M}_{n}^{0}$ due to our assumption $\rho_{n}^{0} > 0$ on the initial data. Next, we turn to the evolution equation satisfied by $1/\rho_{n}$ on $\mathbb{T} \times [0,T^{*})$. Since $\rho_{n} \ne 0$ on this domain, the function $1/\rho_{n}$ possesses the same regularity as $\rho_{n}$ and it is straightforward to verify using the continuity equation that 
\begin{equation}
    \partial_{t} \left( \frac{1}{\rho_{n}} \right) + u_{n} \partial_{x} \left( \frac{1}{\rho_{n}} \right) = \frac{1}{\rho_{n}} \partial_{x} u_{n}.
\end{equation}
Using the definition of $w_{n}$, namely that $w_{n} = u_{n} + \partial_{x}p_{n}(\rho_{n})$, this becomes
\begin{align} \notag
    \partial_{t} \left( \frac{1}{\rho_{n}} \right) &= -u_{n} \partial_{x}\left( \frac{1}{\rho_{n}}  \right) + \frac{1}{\rho_{n}} \partial_{x}w_{n} - \frac{1}{\rho_{n}}\left[ p_{n}''(\rho_{n})|\partial_{x}\rho_{n}|^{2} + p_{n}'(\rho_{n})\partial_{x}^{2}\rho_{n}   \right] \\[1ex] &= -u_{n} \partial_{x}\left( \frac{1}{\rho_{n}}  \right) + \frac{1}{\rho_{n}} \partial_{x}w_{n} - \frac{p_{n}''(\rho_{n})}{\rho_{n}} |\partial_{x}\rho_{n}|^{2} + p_{n}'(\rho_{n})\left[ -2\rho_{n}^{2} \partial_{x} \left( \frac{1}{\rho_{n}} \right)^{2} + \rho_{n} \partial_{x}^{2}\left( \frac{1}{\rho_{n}} \right)  \right], \label{3.4h}
\end{align} where we have used the relations
\begin{align*}
    \partial_{x} \rho_{n} = - \rho_{n}^{2} \partial_{x} \left( \frac{1}{\rho_{n}} \right), ~~
    \partial_{x}^{2} \rho_{n} = 2\rho_{n}^{3} \partial_{x} \left( \frac{1}{\rho_{n}}\right)^{2} - \rho_{n}^{2} \partial_{x}^{2} \left( \frac{1}{\rho_{n}} \right).
\end{align*} In a similar way to the previous argument involving $W_{n}$, we can define $
    P_{n}(x,t) := \frac{1}{\rho_{n}(x,t)}$ and the corresponding maximum function \begin{equation*}
    P_{n}^{M}(t) := \max_{x \in \mathbb{T}}  \frac{1}{\rho_{n}(x,t)} \text{ on } [0,T^{*}).\end{equation*} One can repeat the argument for $W_{n}^{M}$ to conclude that $P_{n}^{M}$ is differentiable a.e. on $[0,T^{*})$. We then have from equation \eqref{3.4h} that for all $t \in [0,T^{*})$ where $P_{n}^{M}$ is differentiable,
\begin{equation}
    \partial_{t}P_{n}^{M} = - u_{n} \partial_{x} P_{n}^{M} + \rho_{n}^{-1} \partial_{x}w_{n} - P_{n}^{M} p_{n}''(\rho_{n})|\rho_{n}^{2} \partial_{x}P_{n}^{M}|^{2} - 2p_{n}'(\rho_{n})\rho_{n}^{2} \left( \partial_{x} P_{n}^{M}  \right)^{2} + \rho_{n} p_{n}'(\rho_{n}) \partial_{x}^{2}P_{n}^{M}.
\end{equation} Using the facts that $\rho_{n}(\cdot, t) > 0$ on $[0,T^{*})$, $\partial_{x}P_{n}^{M}(t) = 0 $ and $ \partial_{x}^{2}P_{n}^{M} \le 0$, this equation implies the inequality
\begin{equation}
    (\partial_{t}P_{n}^{M})(t) \le \frac{\partial_{x}w_{n}}{\rho_{n}}(x,t) \le \mathcal{M}_{n}^{0},~\text{ for a.e. } ~t \in [0,T^{*}),
\end{equation} where we have made use of \eqref{dxwbound}. Integrating in time gives us
\begin{align} \label{ubdensity}
    \rho_{n}(x,t) &\ge \frac{1}{\mathcal{M}_{n}^{0}t + \displaystyle (\inf_{\mathbb{T}} \rho_{n}^{0})^{-1}}.
\end{align}
This lower bound contradicts our assumption \eqref{3.4b} which was that \begin{equation*}
    \inf_{t \in [0,T^{*})} \min_{x \in \mathbb{T}} \rho_{n}(t,x) = 0.
\end{equation*} We conclude that $T^{*} = + \infty$ and thus our solution is global in time. The uniqueness follows from the fact that our local solution was shown to be unique. Indeed, suppose that we have two global regular solutions $(\rho_{n}^{1}, u_{n}^{1})$ and $(\rho_{n}^{2}, u_{n}^{2})$ coinciding on $[0,T_{1}]$ and further assume that they differ almost everywhere on some interval $[T_{1}, T_{2}]$, i.e. that $\|\rho_{n}^{1} - \rho_{n}^{2} \|_{L^{\infty}(T_{1}, T; L^{\infty})} > 0$ for all $T \in (T_{1}, T_{2}]$ (and analogously for $u_{n}^{i}$). Then taking the initial data $(\rho_{n}^{0}, u_{n}^{0}) := (\rho_{n}^{1}(T_{1}), u_{n}^{1}(T_{1}))$, we deduce from Theorem \ref{localexistence} that $(\rho_{n}^{1}, u_{n}^{1}) =(\rho_{n}^{2}, u_{n}^{2})$ on an interval $[T_{1}, T_{1} + \epsilon]$ for some $\epsilon > 0$, which contradicts our assumption.
\end{proof}
\begin{remark} \label{LBRemark}
    If we assume the existence of positive constants $C, \underline{\rho^{0}}$ independent of $n$ such that \linebreak $ \esssup_{\mathbb{T}} \frac{\partial_{x}w_{n}^{0}}{\rho_{n}^{0}} \le C$ and $\rho_{n}^{0} \ge \underline{\rho^{0}} > 0$, then the bound \eqref{ubdensity} gives us the uniform (in $n$) estimate \begin{equation*}
       \rho_{n}(x,t) \ge \frac{1}{Ct + (\underline{\rho^{0}})^{-1}}.  
    \end{equation*}
\end{remark}
\begin{remark}
    Our proof does not rely on the precise form of the offset function $p_{n}$, but it does require that $p_{n}(r)$ is  increasing for $r \in [0, \overline{\rho_{n}}]$. Thus we expect this result to hold for a wide class of offset functions. In particular, our argument works for the choice $p_{\epsilon}(\rho_{\epsilon}) = \epsilon \rho_{\epsilon}^{\gamma}(1-\rho_{\epsilon})^{-\beta}$ taken by Chaudhuri et al in \cite{HCL}.
\end{remark}
\section{Uniform in $n$ estimates} In this section we obtain uniform estimates on the global regular solution $(\rho_{n}, u_{n})$, whose existence was asserted in the previous section.
\subsection{Bounds from the energy estimates} 
Assuming the condition \eqref{thm2RHObounds} and the first bound of \eqref{thm2wn} on the initial data, we have the following uniform estimate by virtue of our energy estimates.
 \begin{align}
    \begin{split} \label{energyunif}
        &\|\rho_{n}\|_{L^{\infty}_{t}L^{1}_{x}} + \|\sqrt{\rho_{n}}w_{n}\|_{L^{\infty}_{t}L^{2}_{x}} + \|H_{n}(\rho_{n})\|_{L^{\infty}_{t}L^{1}_{x}} + \|\sqrt{\rho_{n}}\partial_{x}p_{n}(\rho_{n})\|_{L^{2}_{t}L^{2}_{x}}\le C.
        \end{split}
    \end{align}
Note that the assumption \eqref{thm2RHObounds} implies that $\|H_{n}(\rho_{n}^{0})\|_{L^{1}_{x}}$ is uniformly bounded in $n$. Thus the estimate from Lemma \ref{bdentropy} is also independent of $n$.
\subsection{Bounding the desired velocity}
The goal of this section is to obtain the uniform bound \eqref{thm2wnuniformbound} on the sequence of desired velocities, $w_{n}$. In the remainder of the paper we assume the hypotheses \eqref{thm2RHObounds}-\eqref{thm2wn} and denote by $C$ a positive constant independent of $n$. Our first result is a bound on the spatial derivative of $w_{n}$. 
\begin{proposition} \label{dxwuniform}
    Provided $\gamma_{n} \ge 1$ we have
    \begin{equation}
        \|\partial_{x}w_{n}\|_{L^{\infty}_{t}L^{4/3}_{x}} \le C.
    \end{equation}
\end{proposition}
\begin{proof}
   It follows from \eqref{3.4f} that the potential $W_{n} = \rho_{n}^{-1}\partial_{x}w_{n}$ satisfies 
    \begin{equation}
    \partial_{t}(\rho_{n}W_{n}) + \partial_{x}(\rho_{n}W_{n}u_{n}) = 0,
    \end{equation} due to the continuity equation. Thus $(\rho_{n}, W_{n}, u_{n})$ satisfies the original momentum equation \eqref{AR-2}. Multiplying by $W_{n}$ and integrating by parts over space and time leads to
    \begin{equation} \label{conservationdxwn}
        \int_{\mathbb{T}} \frac{(\partial_{x}w_{n})^{2}}{\rho_{n}}(x,t)~dx = \int_{\mathbb{T}} \frac{(\partial_{x}w_{n}^{0})^{2}}{\rho_{n}^{0}}(x) ~dx  =: \mathcal{D}_{n}^{0}.
    \end{equation}
    Noticing that by Young's inequality \begin{equation*}
        |\partial_{x}w_{n}|^{4/3} = \left|\frac{\partial_{x}w_{n}}{\sqrt{\rho_{n}}} \right|^{4/3} |\sqrt{\rho_{n}}|^{4/3} \le \frac{2}{3}\left|\frac{\partial_{x}w_{n}}{\sqrt{\rho_{n}}} \right|^{2} + \frac{1}{3}|\rho_{n}|^{2},
    \end{equation*} we infer from \eqref{conservationdxwn} that \begin{equation} \label{dxwunif1}
\|\partial_{x}w_{n}\|_{L^{\infty}_{t}L^{4/3}_{x}}^{3/4} \le \mathcal{D}_{n}^{0} + \|\rho_{n}\|_{L^{\infty}_{t}L^{2}_{x}}^{2}.
    \end{equation} Next, recall from \eqref{energyunif} we have that \begin{equation} \label{Hnbound}
        \|H_{n}(\rho_{n})\|_{L^{\infty}_{t}L^{1}_{x}} = \frac{1}{\gamma_{n}+1}\esssup_{t \in [0,T]} \int_{\mathbb{T}} \rho_{n}^{\gamma_{n}+1}~dx  =\frac{1}{\gamma_{n}+1}\|\rho_{n}\|_{L^{\infty}_{t}L^{\gamma_{n}+1}_{x}}^{\gamma_{n}+1} \le C. \end{equation}
        Therefore the assumption $\gamma_{n} \ge 1$ guarantees that $\|\rho_{n}\|_{L^\infty_{t}L^{2}_{x}} \le C$ (uniformly in $n$). Returning to \eqref{dxwunif1}, we can obtain the required bound.
   \end{proof}
   \begin{remark}
       The exponent $4/3$ is taken for convenience. A similar argument can be used to show $\|\partial_{x}w_{n}\|_{L^{\infty}_{t}L^{p}_{x}} \le C$ for any $p \in [1,2)$, assuming $\gamma_{n}$ is sufficiently large.
   \end{remark}
   We now obtain a bound for the sequence of desired velocities $w_{n}$.
\begin{proposition} \label{unifw}
The desired velocity $w_{n}$ satisfies \begin{equation*}
         \|w_{n}\|_{L^{\infty}_{t}W^{1,\frac{4}{3}}_{x}}\le C.
    \end{equation*} 
\end{proposition}
\begin{proof}
Taking $r = \rho_{n}$ in the generalised Poincare inequality (see Proposition \ref{poincare} in the Appendix) gives us
\begin{align} \label{wnL1a}
    \|w_{n}\|_{L^{1}_{x}} \le C \left( \|\partial_{x}w_{n}\|_{L^{1}_{x}} + \int_{\mathbb{T}} \rho_{n}w_{n}~dx\right).
\end{align} Using Young's inequality and Proposition \ref{intermediatebounds}, \begin{equation*}
    \|\rho_{n}w_{n}\|_{L^{\infty}_{t}L^{1}_{x}} \le \|\sqrt{\rho_{n}}w_{n}\|_{L^{\infty}_{t}L^{2}_{x}} + \|\rho_{n}\|_{L^{\infty}_{t}L^{1}_{x}} \le C.\end{equation*} Then taking the supremum over time in \eqref{wnL1a} and using Proposition \ref{dxwuniform} gives us $\|w_{n}\|_{L^{\infty}_{t}L^{1}_{x}} \le C$. Combining this estimate with Proposition \ref{dxwuniform} leads to the result.

\end{proof} We now make note of some bounds which will be used to prove the remaining estimates in this section.
\begin{corollary} \label{intermediatebounds}
    Provided $\gamma_{n} \ge 1$ we have  \begin{align}
        &\|\sqrt{\rho_{n}}u_{n}\|_{L^{2}_{t,x}} + \|\rho_{n}\|_{L^{\infty}_{t}L^{2}_{x}} \le C,  \label{sqrtrhouL22} \\[1ex]
        &\|w_{n}\|_{L^{\infty}_{t,x}} +\|\rho_{n}w_{n}\|_{L^{2}_{t,x}} \le C,  \label{wLinf}
    \end{align}
\end{corollary}
\begin{proof}
    The first estimate in \eqref{sqrtrhouL22} follows from \eqref{energyunif} and the relationship $\sqrt{\rho_{n}}u_{n} = \sqrt{\rho_{n}}w_{n} - \sqrt{\rho_{n}}\partial_{x}p_{n}(\rho_{n})$, while the second was obtained from \eqref{Hnbound}. The $L^{\infty}_{t,x}$ bound on $w_{n}$ follows from the embedding $W^{1,1}(\mathbb{T})  \hookrightarrow L^{\infty}(\mathbb{T})$. For the final bound, we have $\|\rho_{n}w_{n}\|_{L^{2}_{t,x}} \le t \|w_{n}\|_{L^{\infty}_{t,x}} \|\rho_{n}\|_{L^{\infty}_{t}L^{2}_{x}} \le C$, thanks to \eqref{sqrtrhouL22}. 
\end{proof}
\subsection{An improved estimate for the potential}
We now obtain a uniform bound on the potential $\pi_{n}$ which is defined by the relation \eqref{pi1}. It will be useful to note that
\begin{equation}
    \pi_{n}(\rho_{n}) = \frac{\gamma_{n}}{\gamma_{n}+1}\rho_{n}^{\gamma_{n}+1} .\label{pi2}
\end{equation}
\begin{lemma} \label{dxpiunif}
The potential $\pi_{n}$ satisfies
\begin{equation} \label{dxpiunifeqn}
    \|\partial_{x}\pi_{n}\|_{L^{2}_{t,x}} \le C.
\end{equation}
\end{lemma}
\begin{proof}
    First note that the continuity equation can be expressed as
    \begin{equation} \label{continuitywn}
        \partial_{t}\rho_{n} + \partial_{x}(\rho_{n}w_{n}) = \partial_{x}^{2}\pi_{n}.
    \end{equation} Multiplying by $\pi_{n}$ and integrating over $[0,t] \times \mathbb{T}$ and by parts, we find
    \begin{align*}
      \frac{1}{\gamma_{n}+2} \int_{\mathbb{T}} - \rho_{n}^{\gamma_{n}+2}(x,t) + (\rho_{n}^{0})^{\gamma_{n}+2}(x)~dx   + \int_{0}^{t}\int_{\mathbb{T}} \rho_{n}w_{n} \partial_{x}\pi_{n} ~dxds =  \int_{0}^{t}\int_{\mathbb{T}} |\partial_{x}\pi_{n}|^{2}~dxds.
    \end{align*} Since $\rho_{n} > 0$ this reduces to
    \begin{equation}
        \int_{0}^{t}\int_{\mathbb{T}} |\partial_{x}\pi_{n}|^{2}~dxds \le \int_{0}^{t}\int_{\mathbb{T}} \rho_{n}w_{n} \partial_{x}\pi_{n} ~dxds + \frac{1}{\gamma_{n}+2} \int_{\mathbb{T}}(\rho_{n}^{0})^{\gamma_{n}+2}(x)~dx.
    \end{equation} Note that the assumption \eqref{thm2RHObounds} implies that the final term on the right hand-side is uniformly bounded in $n$. Additionally, using Young's inequality with the first term on the right hand-side leads to
    \begin{equation*}
        \frac{1}{2}\int_{0}^{t}\int_{\mathbb{T}} |\partial_{x}\pi_{n}|^{2}~dxds \le C\|\rho_{n}w_{n}\|_{L^{2}_{t,x}}^{2} + C.
    \end{equation*} The result follows from \eqref{wLinf}.
\end{proof}
\begin{corollary} \label{intermediatebounds2}
    We have \begin{align} \label{rhouunif}
        &\|\rho_{n}u_{n}\|_{L^{2}_{t,x}} \le C, \\[1ex] 
        &\|\sigma \sqrt{\rho_{n}}u_{n}\|_{L^{\infty}_{t}L^{2}_{x}} \le C, ~ \text{ where } \sigma(t) = \min(1,t). \label{sigmarhou}
    \end{align}
\end{corollary}
\begin{proof}
    Since $\rho_{n}u_{n} = \rho_{n}w_{n} - \partial_{x}\pi_{n}$, \eqref{rhouunif} follows from \eqref{wLinf} and \eqref{dxpiunifeqn}. To acquire \eqref{sigmarhou}, we multiply the momentum equation \eqref{a2} by $\sigma u_{n}$. Integrating by parts in space and time leads to
    \begin{align*}
        \int_{\mathbb{T}} \sigma \rho_{n}u_{n}^{2}(t,x) - \sigma \rho_{n}u_{n}^{2}(t,0)~dx - \int_{0}^{t} \int_{\mathbb{T}} \mathbbm{1}_{[0,1]}(s)\rho_{n}u_{n}^{2}~dxds &+ \int_{0}^{t} \int_{\mathbb{T}} \sigma \lambda_{n}(\rho_{n}(\partial_{x}u_{n})^{2}~dxds \\[1ex]
         &= \int_{0}^{t} \int_{\mathbb{T}} \sigma \rho_{n}u_{n} ( \partial_{t}u_{n} + u_{n}\partial_{x}u_{n})~dxds.
    \end{align*} The term on the right hand-side disappears once we integrate by parts and use the continuity equation. Also notice that $\sigma(0) = 0$. Therefore we have \begin{equation*}
         \int_{\mathbb{T}} \sigma \rho_{n}u_{n}^{2}(t,x)~dx \le \frac{1}{2}\int_{0}^{t} \int_{\mathbb{T}} \mathbbm{1}_{[0,1]}(s)\rho_{n}u_{n}^{2}~dxds \le \|\sqrt{\rho_{n}}u_{n}\|_{L^{2}_{t,x}} \le C,
    \end{equation*}thanks to \eqref{sqrtrhouL22}.
\end{proof}
\begin{lemma} \label{piunif}
    Under the assumptions \eqref{thm2RHObounds}- \eqref{thm2wn}, the potential $\pi_{n}$ satisfies
    \begin{equation} \label{unifpi}
        \|\pi_{n}\|_{L^{1}_{t,x}} \le C.
    \end{equation}
\end{lemma}
\begin{proof} We begin by defining the function $\psi_{n} : \mathbb{T} \times [0,T] \to \mathbb{R}$ given by
    \begin{equation}
        \psi_{n}(x,t) =  \int_{0}^{x} \left( \rho_{n}^{0}(y) - \langle \rho_{n} \rangle \right)~dy - \int_{0}^{t} \rho_{n}u_{n} (x,s)~ds,
    \end{equation} where $\langle \rho_{n} \rangle := |\mathbb{T}|^{-1}\int_{\mathbb{T}} \rho_{n}(t,y)~dy$.
    Assuming that $\mathbb{T} = [0,1]$, we can verify some key properties of $\psi_{n}$. Firstly, we can show that $\psi_{n}$ is periodic in space. Indeed, we have
      $ \displaystyle \psi_{n}(0,t) = - \int_{0}^{t} \rho_{n}u_{n} (0,s)~ds$  while \begin{align*}
            \psi_{n}(1,t) &= \int_{\mathbb{T}} \rho_{n}^{0}(y) - \langle \rho_{n} \rangle~dy - \int_{0}^{t} \rho_{n}u_{n} (1,s)~ds = - \int_{0}^{t} \rho_{n}u_{n} (1,s)~ds \\[1ex] &= - \int_{0}^{t} \rho_{n}u_{n} (0,s)~ds = \psi_{n}(0,t).
        \end{align*}
        Next, we can directly compute $\partial_{x}\psi_{n}$ to find
        \begin{align*}
            \partial_{x}\psi_{n} &= \rho_{n}^{0}(x) - \langle \rho_{n} \rangle - \int_{0}^{t} \partial_{x} (\rho_{n}u_{n})(x,s)~ds  \\[1ex] &= \rho_{n}^{0}(x) - \langle \rho_{n} \rangle + \rho_{n}(x,t) - \rho_{n}(x,0) = \rho_{n} - \langle \rho_{n} \rangle,
        \end{align*} where we have made use of the continuity equation. Then by virtue of \eqref{sqrtrhouL22} we have  $\|\partial_{x}\psi_{n}\|_{L^{\infty}_{t}L^{2}_{x}} \le C$. In order to obtain a bound on $\psi_{n}$, we note that $\| \rho_{n}^{0} - \langle \rho_{n} \rangle \|_{L^{\infty}_{t,x}} \le C$ from assumption \eqref{thm2RHObounds} and so
        \begin{align*}
            \|\psi_{n}(t)\|_{L^{2}_{x}} &\le C + \int_{\mathbb{T}} \left( \int_{0}^{t} \rho_{n}u_{n}(x,s)~ds \right)^{2}~dx \\[1ex] &\le C + \|\rho_{n}u_{n}\|_{L^{2}_{t,x}}^{2} \le C,
        \end{align*} due to Jensen's inequality and \eqref{rhouunif}. Thus we also have $\|\psi_{n}\|_{L^{\infty}_{t}L^{2}_{x}} \le C$. From the Sobolev inequality \eqref{Sinf} we can now deduce that $\|\psi_{n}\|_{L^{\infty}_{t,x}} \le C$. We also note that $\partial_{t}\psi_{n} = - \rho_{n}u_{n}$ since $\langle \rho_{n} \rangle$ is constant as a function of time (and space). Therefore from \eqref{rhouunif} we have $\|\partial_{t}\psi_{n}\|_{L^{2}_{t,x}} \le C$. Next, we define the test function $\zeta_{n} : \mathbb{T} \times [0,T] \to \mathbb{R}$ as 
        \begin{equation}
            \zeta_{n}(x,t) := \int_{0}^{x} \left(\psi_{n}(y,t) - \frac{1}{|\mathbb{T}|}\int_{\mathbb{T}} \psi_{n}(z,t)~dz \right)dy.
        \end{equation} Note that $\zeta_{n}$ is periodic in space. Using the properties of $\psi_{n}$ it is also straightforward to verify that $\|\zeta_{n}\|_{L^{\infty}_{t,x}} + \|\partial_{x}\zeta_{n}\|_{L^{\infty}_{t,x}} \le C$. We now obtain a bound on $\partial_{t}\zeta_{n}$. Note that
        \begin{align*}
            \partial_{t}\zeta_{n}(x,t) = \int_{0}^{x}\left(-\rho_{n}u_{n}(y,t) + |\mathbb{T}|^{-1}\int_{\mathbb{T}}\rho_{n}u_{n}(z,t)~dz\right)dy
        \end{align*} and so $|\partial_{t}\zeta_{n}|^{2} \le 2 \|\rho_{n}u_{n}(t)\|_{L^{2}_{x}}^{2} \le C$
         thanks to Jensen's inequality and  \eqref{rhouunif}. Integrating over space and time in this inequality gives $\|\partial_{t}\zeta_{n}\|_{L^{2}_{t,x}} \le C$. Note that $\partial_{x}^{2}\zeta_{n} = \rho_{n} - \langle \rho_{n} \rangle$ by construction. Multiplying the continuity equation \eqref{continuitywn} by $\zeta_{n}$ and integrating by parts in space and time, we therefore have
         \begin{align*}
           -  \int_{0}^{t}\int_{\mathbb{T}} \rho_{n} \partial_{t}\zeta_{n} ~dxds + \int_{\mathbb{T}} \zeta_{n}\rho_{n}(x,t) &- \zeta_{n}\rho_{n}(x,0)~dx - \int_{0}^{t}\int_{\mathbb{T}} \partial_{x}\zeta_{n} \rho_{n}w_{n}~dxds \\[1ex] &= \int_{0}^{t}\int_{\mathbb{T}} (\rho_{n} - \langle \rho_{n} \rangle ) \pi_{n}~dxds.
         \end{align*} The absolute value of the left side of the above equation can be bounded independently of $n$ due to the properties of $\zeta_{n}$ already established as well as Corollary \ref{intermediatebounds}. Therefore we find
         \begin{equation} \label{pimean}
            \left| \int_{0}^{t}\int_{\mathbb{T}} (\rho_{n} - \langle \rho_{n} \rangle ) \pi_{n}~dxds \right| \le C.
         \end{equation}  We now wish to extract a bound on the quantity $\|\pi_{n}\|_{L^{1}_{t,x}}$. To this end, we define \begin{equation*}
     S_{m} := \frac{1 + \langle \rho_{n} \rangle }{2}
 \end{equation*} and consider the regions $\left\{ \rho \le S_{m} \right\}$ and $\left\{ \rho > S_{m} \right\}$ separately. Firstly, we have that
 \begin{equation} \label{pi8}
     \left|  \int_{0}^{t}\int_{\mathbb{T}} \pi_{n}(x,t) \mathbbm{1}_{ \left\{ \rho \le S_{m} \right\} } ~dx ds\right| \le C
 \end{equation} by the observation that when $ \rho_{n} < S_{m} < 1$, we have $\pi_{n}(\rho_{n}) = \frac{\gamma_{n}}{\gamma_{n}+1}\rho_{n}^{\gamma_{n}+1} \to 0$ as $n \to \infty$.  This in particular implies that
 \begin{align*}
     & \left| \int_{0}^{t}\int_{\mathbb{T}} (\rho_{n} - \langle \rho_{n} \rangle ) \pi_{n}(x,t)  \mathbbm{1}_{ \left\{ \rho > S_{m} \right\}  }~dxds \right| \\[1ex] &\hspace{0.5cm}\le  \left| \int_{0}^{t}\int_{\mathbb{T}} (\rho_{n} - \langle \rho_{n} \rangle ) \pi_{n}(x,t) \mathbbm{1}_{ \left\{ \rho \le S_{m} \right\} }~dxds \right| +  \left|\int_{0}^{t}\int_{\mathbb{T}} (\rho_{n} - \langle \rho_{n} \rangle ) \pi_{n}(x,t)~dxds \right| \le C,
  \end{align*} due to \eqref{pi8} and \eqref{pimean}. When $\rho_{n} \ge S_{m}$, the assumption $\langle \rho_{n} \rangle \le \hat{\rho} < 1$ implies that $\rho_{n} - \langle \rho_{n} \rangle \ge \frac{1 - \langle \rho_{n} \rangle }{2} \ge \frac{1 - \hat{\rho} }{2} > 0$. Thus, we have
  \begin{equation*}
      C \ge \left| \int_{0}^{t}\int_{\mathbb{T}} (\rho_{n} - \langle \rho_{n} \rangle ) \pi_{n}(x,t) \mathbbm{1}_{\left\{ \rho > S_{m} \right\}}~dxds \right| \ge \frac{1 - \hat{\rho}}{2} \left|  \int_{0}^{t}\int_{\mathbb{T}} \pi_{n}(x,t) \mathbbm{1}_{\left\{ \rho > S_{m} \right\}}~dxds \right|,
  \end{equation*} and so we conclude from this and \eqref{pi8} that for any $t \in [0,T]$,
  \begin{equation} \label{lambdadxu}
    \|\pi_{n}\|_{L^{1}_{t,x}} =   \int_{0}^{t}\int_{\mathbb{T}}  \pi_{n}(x,t) ~dxds  \le C.
  \end{equation}
\end{proof} We can improve this bound via a Hoff-type estimate \cite{hoff_global_1998, hoff_discts_1997}, which also takes inspiration from \cite{Burtea_2020}.
\begin{corollary} For any $\tau \in (0,T)$, we have
   \begin{equation}
        \|\pi_{n}\|_{L^{\infty}([\tau, T];L^{1}(\mathbb{T}))} \le C.
    \end{equation}
\end{corollary}
\begin{proof}
   We multiply the momentum equation \eqref{a2} by $\sigma(t)\psi_{n}$, where $\sigma(t) := \min(1,t)$. Integrating by parts in space and time and simplifying,
   \begin{align}
   \begin{aligned}
       - \int_{0}^{t}\int_{\mathbb{T}}  \mathbbm{1}_{[0,1]}(s)\psi_{n} \rho_{n}u_{n}~dxds + \int_{\mathbb{T}}\sigma(t)\psi_{n}\rho_{n}u_{n}(x,t)~dx + \langle \rho_{n} \rangle \int_{0}^{t}\int_{\mathbb{T}} \sigma(s) \rho_{n}u_{n}^{2}~dxds \\[1ex] = - \int_{0}^{t}\int_{\mathbb{T}} \sigma(s)(\rho_{n} - \langle \rho_{n} \rangle ) \lambda_{n}(\rho_{n})\partial_{x}u_{n}~dxds.
        \end{aligned}
   \end{align}
   Using Corollaries \ref{intermediatebounds}, \ref{intermediatebounds2} and the regularity of $\psi_{n}$ we obtain \begin{align} \label{pisigmalambda}
       \left| \int_{0}^{t}\int_{\mathbb{T}} \sigma(s)(\rho_{n} - \langle \rho_{n} \rangle ) \lambda_{n}(\rho_{n})\partial_{x}u_{n}~dxds \right| \le C.
   \end{align}  Next, we derive the evolution equation for $(\rho_{n} - \langle \rho_{n} \rangle )\pi_{n}$ which reads
    \begin{equation} \label{RCE}
         \partial_{t} \left[(\rho_{n} - \langle \rho_{n} \rangle )\pi_{n}\right] + \partial_{x} \left[(\rho_{n} - \langle \rho_{n} \rangle ) \pi_{n} u_{n} \right] +  \left[ \rho_{n}(\rho_{n} - \langle \rho_{n} \rangle )\pi_{n}' + \langle \rho_{n} \rangle \pi_{n}  \right]\partial_{x} u_{n} = 0. 
    \end{equation} Multiplying by $\sigma(\cdot)$ and integrating over space and time,
    \begin{align}
    \begin{aligned} \label{sigmaRCE}
\int_{0}^{t}\int_{\mathbb{T}}\sigma(s)\partial_{s}\left[(\rho_{n} - \langle \rho_{n} \rangle )\pi_{n}\right]~dxds = &-\int_{0}^{t} \int_{\mathbb{T}} \sigma(s) (\rho_{n} - \langle \rho_{n} \rangle) \lambda_{n}(\rho_{n}) \partial_{x}u_{n}~dxds \\[1ex] &- \langle \rho_{n} \rangle \int_{0}^{t} \int_{\mathbb{T}} \pi_{n} \partial_{x} u_{n}~dxds.
    \end{aligned}
    \end{align}
    The first term on the RHS is bounded independently of $n$ thanks to \eqref{pisigmalambda}. To bound the final term, we can integrate by parts to see that
    \begin{align*}
        \int_{0}^{t}\int_{\mathbb{T}} \pi_{n}\partial_{x}u_{n}~dxds = - \int_{0}^{t}\int_{\mathbb{T}} u_{n} \partial_{x}\pi_{n}~dxds = -\int_{0}^{t}\int_{\mathbb{T}} u_{n} (\rho_{n}w_{n} - \rho_{n}u_{n})~dxds \le C,
    \end{align*} thanks to Corollaries \ref{intermediatebounds} and \ref{intermediatebounds2}. Therefore returning to \eqref{sigmaRCE} we find
    \begin{equation*}
\int_{0}^{t}\int_{\mathbb{T}}\sigma(s)\partial_{s}\left[(\rho_{n} - \langle \rho_{n} \rangle )\pi_{n}\right]~dxds \le C.
    \end{equation*} Integrating by parts,
    \begin{align*}
        \sigma(t)\int_{\mathbb{T}} (\rho_{n} - \langle \rho_{n} \rangle ) \pi_{n}(x,t)~dx \le C + \int_{0}^{t}\int_{\mathbb{T}} \mathbbm{1}_{[0,1]}(s) ( \rho_{n} - \langle \rho_{n} \rangle ) \pi_{n}(x,s)~dxds.
    \end{align*} Recalling the bound \eqref{pimean} we have for all $t \in [0,T]$
    \begin{equation*}
        \sigma(t)\int_{\mathbb{T}} (\rho_{n} - \langle \rho_{n} \rangle ) \pi_{n}(x,t)~dx \le C.
    \end{equation*} Considering the regions $\left\{ \rho \le S_{m} \right\}$ and $\left\{ \rho > S_{m} \right\}$ separately and repeating the same argument which was seen in the proof of Lemma \ref{piunif}, we get \begin{equation} \label{sigmapiLINF}
        \sigma(t)\int_{\mathbb{T}} \pi_{n}(x,t)~dx \le C,
    \end{equation} from which the bound follows.
\end{proof}
\section{The limit passage}
The aim of this section is to complete the proof of Theorem \ref{limitexistence}.
\begin{proof}[Proof of Theorem \ref{limitexistence}]
For $n$ fixed, we have proved the existence of a unique regular solution $(\rho_{n}, w_{n})$ to 
the system \eqref{limsystem1}-\eqref{limsystem2} on $\mathbb{T} \times [0,T]$ with $w_{n} = u_{n} + \partial_{x}p_{n}$. Multiplying \eqref{limsystem1}-\eqref{limsystem2} by $\phi \in C^{\infty}_{c}([0,T] \times \mathbb{T})$ and integrating by parts, an integral formulation for the above system is given by
 \begin{subequations} 
\newcommand{\mystrut}{\vphantom{\pder{}{}}}
\begin{numcases}{}  \label{HCM-WF1n}
   -  \int_{0}^{t} \int_{\mathbb{T}} \rho_{n} \partial_{t} \phi ~dxds + \int_{\mathbb{T}} \rho_{n}\phi(x,t) - \rho_{n}^{0}\phi(x,0)~dx + \int_{0}^{t} \int_{\mathbb{T}} \partial_{x}\pi_{n} \partial_{x} \phi ~dxds= \int_{0}^{t} \int_{\mathbb{T}} \rho_{n} w_{n} \partial_{x} \phi ~dxds, \\[1ex]
   -  \int_{0}^{t} \int_{\mathbb{T}} \rho_{n} w_{n} \partial_{t}\phi ~dxds + \int_{\mathbb{T}}\rho_{n} w_{n} \phi(x,t) - \rho_{n}^{0} w_{n}^{0}\phi (x,0)~dx +  \int_{0}^{t} \int_{\mathbb{T}} (w_{n} \partial_{x} \pi_{n} - \rho_{n} w_{n}^{2})\partial_{x}\phi ~dxds =  0. \label{HCM-WF2n}
\end{numcases}
\end{subequations}
 The estimates obtained in the previous section imply that there exists a triple $(\rho, w, \pi)$ such that, up to the extraction of a subsequence, we have
\begin{align}                   \label{RHOW}
    &\rho_{n} \rightharpoonup^{*} \rho \text{ weakly-* in }L^{\infty}(0,T; L^{2}(\mathbb{T})), \\[1ex]        \label{WWSTAR}
    &w_{n} \rightharpoonup^{*} w \text{ weakly-* in } L^{\infty}((0,T) \times \mathbb{T}), \\[1ex] \label{PIWSTAR}
    &\pi_{n} \rightharpoonup^{*} \pi \text{ weakly-* in } \mathcal{M}((0,T) \times \mathbb{T}), \\[1ex]
    \label{DXPIWSTAR}
    &\partial_{x}\pi_{n} \rightharpoonup^{*} \partial_{x} \pi \text{ weakly in } L^{2}(0,T; L^{2}(\mathbb{T})),     
\end{align} Thanks to the estimate \eqref{sigmapiLINF}, the bound $\|\partial_{x}(\sigma \pi_{n})\|_{L^{2}_{t,x}} \le C$ and the embedding $W^{1,2}(\mathbb{T}) \hookrightarrow L^{\infty}(\mathbb{T})$, we also have
\begin{equation} \label{SIGMAPI}
\sigma\pi_{n} \rightharpoonup \sigma\pi \text{ weakly-* in } L^{2}(0,T; L^{\infty}(\mathbb{T})), 
\end{equation}where again $\sigma(t) = \min(1,t)$.
The key tool we will need in order to complete the limit passage is the strong convergence of $w_{n}$. Since $\|u_{n}\partial_{x}w_{n}\|_{L^{1}_{t,x}} \le \|\sqrt{\rho_{n}}u_{n}\|_{L^{2}_{t,x}}\| \rho_{n}^{-1/2}\partial_{x}w_{n}\|_{L^{2}_{t,x}} \le C$, the momentum equation $\partial_{t}w_{n} + u_{n}\partial_{x}w_{n}=0$ implies that $
    \|\partial_{t}w_{n}\|_{L^{1}_{t,x}} \le C.$
 This observation along with the chain of embeddings $W^{1,\frac{4}{3}} \hookrightarrow \hookrightarrow  L^{\infty} \hookrightarrow L^{1}$ allows us to use the Aubin-Lions lemma to deduce that
\begin{equation} \label{strongw}
    w_{n} \to w \text{ strongly in } L^{\infty}((0,T) \times \mathbb{T}).
\end{equation} Consequently, the following convergences hold:
\begin{align}
    \rho_{n}w_{n} &\rightharpoonup^{*} \rho w \text{ weakly-* in }L^{\infty}(0,T; L^{2}(\mathbb{T})),  \label{rhow} \\[1ex]
    \rho_{n}w_{n}^{2} &\rightharpoonup^{*} \rho w^{2} \text{ weakly-* in }L^{\infty}(0,T; L^{2}(\mathbb{T})), \label{rhow2}
\end{align} thanks to \eqref{RHOW}. Due to the hypotheses \eqref{idconvergence1} and \eqref{idconvergence2}, we have $\rho_{n}^{0} \rightharpoonup \rho^{0}$ and $\rho_{n}^{0}w_{n}^{0} \rightharpoonup \rho^{0}w^{0}$ weakly in $L^{1}(\mathbb{T})$. Therefore as $n \to \infty$ we have for all $t \in [0,T]$,
\begin{align} \label{WFconv1b}
    \int_{\mathbb{T}} \rho_{n}\phi(x,t) - \rho_{n}^{0}\phi(x,0)~dx &\to \int_{\mathbb{T}} \rho \phi(x,t) - \rho^{0}\phi(x,0)~dx, \\[1ex]
    \int_{\mathbb{T}} \rho_{n}w_{n}\phi(x,t) - \rho_{n}^{0}w_{n}^{0}\phi(x,0)~dx &\to \int_{\mathbb{T}} \rho w \phi(x,t) - \rho^{0}w^{0}\phi(x,0)~dx. \label{WFconv2b}
\end{align} The convergences \eqref{RHOW}-\eqref{WFconv2b} are sufficient to pass to the limit in \eqref{HCM-WF1n} and \eqref{HCM-WF2n} to obtain \eqref{HCM-WF1} and \eqref{HCM-WF2} respectively. It remains to verify \eqref{HCM-WF3}. Due to the assumption \eqref{thm2RHObounds} it is easy to see that the limit $\rho$ satisfies $0 \le \rho$. To show $\rho \le 1$, we proceed in a similar way to that which can be seen in Theorem 4.1. of \cite{Lions1999}.
\begin{proposition}
    The density $\rho$ satisfies $0 \le \rho \le 1$.
\end{proposition}\begin{proof}
    For $1 < p < \infty$ such that $\gamma_{n} > p$, the embedding of $L^{p}$ spaces implies that
\begin{equation*}
    \|\rho_{n}\|_{L^{\infty}_{t}L^{p}_{x}} \le \|\rho_{n}\|_{L^{\infty}_{t}L^{1}_{x}}^{\frac{a_{n}}{p}} \|\rho_{n}\|_{L^{\infty}_{t}L^{\gamma_{n}+1}_{x}}^{\frac{1-a_{n}}{p}}
\end{equation*} where $a_{n} = \frac{p-\gamma_{n}}{1-\gamma_{n}} \to 1$ as $n \to \infty$. We now recall the bound
\begin{align*}
   \|\rho_{n}\|_{L^{\infty}_{t}L^{\gamma_{n}+1}_{x}} \le \left(C(\gamma_{n}+1)\right)^{\frac{1}{\gamma_{n}+1}}
\end{align*} which follows from \eqref{Hnbound}. Thus, we have
\begin{equation*}
    \|\rho_{n}\|_{L^{\infty}_{t}L^{p}_{x}} \le C^{\frac{a_{n}}{p}} (C(\gamma_{n}+1))^{\frac{1-a_{n}}{\gamma_{n}+1}} =: C^{\frac{a_{n}}{p}} \tilde{C}_{n},
\end{equation*} and it is easy to see that $\tilde{C}_{n} \to 1$ as $n \to \infty$. Using Fatou's lemma,
\begin{equation*}
      \|\rho\|_{L^{\infty}_{t}L^{p}_{x}} \le \liminf_{n \to \infty} \|\rho_{n}\|_{L^{\infty}_{t}L^{p}_{x}} \le C^{\frac{1}{p}}.
\end{equation*} Taking $p \to \infty$ in this inequality rewards us with
\begin{equation*} \vspace{-0.3cm}
    \|\rho\|_{L^{\infty}_{t,x}} \le \liminf_{p \to \infty} \|\rho\|_{L^{\infty}_{t}L^{p}_{x}} \le 1.
\end{equation*}\end{proof} Next, we estimate $\partial_{t}\rho_{n}$ and adopt the notation for the duality bracket $\langle \cdot, \cdot \rangle_{*} := \langle \cdot, \cdot \rangle_{(W^{1,2}_{0})^{*} \times W^{1,2}_{0}}$. For an arbitrary $f \in W^{1,2}_{0}(\mathbb{T})$, we have
\begin{align*}
    \langle\partial_{t} \rho_{n}(t), f \rangle_{*} &= \langle - \partial_{x} (\rho_{n}(t) u_{n}(t)), f \rangle_{*} = (\rho_{n}(t) u_{n}(t), \partial_{x}f)_{L^{2}_{x}} \\[1ex] &\le \|\rho_{n}u_{n}(t)\|_{L^{2}_{x}} \|f\|_{W^{1, 2}_{0}}.
\end{align*} It follows from this that
\begin{equation} \label{rhonegativesobolev}
    \|\partial_{t} \rho_{n}\|_{L^{2}_{t}W^{-1,2}_{x}} \le \|\rho_{n}u_{n}\|_{L^{2}_{t,x}} \le C.
\end{equation} We also know $\|\partial_{x}(\sigma \pi_{n})\|_{L^{2}_{t,x}} \le C$ and that $\sigma \pi_{n}, \rho_{n}$ converge weakly in $L^{2}_{t,x}$ to $\sigma \pi, \rho$ respectively. We can therefore invoke a standard compensated compactness argument (see Lemma 5.1 of \cite{mathlions}) to get 
\begin{equation} \label{1-rhodistr}
    (1-\rho_{n})\sigma\pi_{n} \longrightarrow (1-\rho)\sigma\pi, ~ \text{ in } \mathcal{D}'((0,T) \times \mathbb{T}).
\end{equation}
On the other hand, we have
\begin{align*}
    (1-\rho_{n})\sigma\pi_{n} &= \frac{\gamma_{n}}{\gamma_{n}+1}(1-\rho_{n}) \sigma\rho_{n}^{\gamma_{n}+1} \\[1ex]
    &= \frac{\gamma_{n}}{\gamma_{n}+1} \left((1-\rho_{n}) \sigma \rho_{n}^{\gamma_{n}+1} \mathbbm{1}_{\left\{ (x,t) ~ : ~ 0 < \rho_{n}(t,x) \le 1  \right\}} + (1-\rho_{n}) \sigma \rho_{n}^{\gamma_{n}+1} \mathbbm{1}_{\left\{ (x,t) ~ : ~  \rho_{n}(t,x) > 1  \right\}} \right) \\[1ex] &=: A_{n} + B_{n}.
\end{align*}
It is straightforward to see that $A_{n} \to 0$ a.e. as $n \to \infty$. Thanks to Lemma \ref{piunif}, we can justify that
\begin{align*}
    \|B_{n}\|_{L^{1}_{t,x}} &= \int_{0}^{T} \int_{\mathbb{T}} |(1-\rho_{n})\sigma\pi_{n}| \mathbbm{1}_{\left\{ (x,t) ~ : ~  \rho_{n}(t,x) > 1  \right\}} ~dxds \\[1ex] &\le \|1-\rho_{n}\|_{L^{\infty}_{t}L^{1}_{x}}\|\pi_{n}\|_{L^{1}_{t}L^{\infty}_{x}} \cdot \mu\left( \left\{ \rho_{n} > 1 \right\} \right) \\[1ex] &\le C \cdot   \mu\left( \left\{ \rho_{n} > 1 \right\} \right) \longrightarrow 0,
\end{align*} as $n \to \infty$, where $\mu$ represents the Lebesgue measure. Note that $\|\pi_{n}\|_{L^{1}_{t}L^{\infty}_{x}} \le C$ follows from Lemmas \ref{dxpiunif}, \ref{piunif} and the embedding $W^{1,1}(\mathbb{T}) \hookrightarrow L^{\infty}(\mathbb{T})$. Therefore we have that
\begin{equation} 
    (1-\rho_{n})\sigma\pi_{n} \to 0 \text{ strongly in } L^{1}(0,T; L^{1}(\mathbb{T})).
\end{equation} Together with \eqref{1-rhodistr}, this allows us to deduce that
\begin{equation}
    (1-\rho)\sigma\pi = 0
 \text{ a.e. in } \mathbb{T} \times [0,T]. \end{equation} Since $\sigma(t) = \min(1,t)$ vanishes only on a set of measure zero $(\text{i.e. on } \{t=0\} \times \mathbb{T})$, we in fact have \begin{equation}
    (1-\rho)\pi = 0
 \text{ a.e. in } \mathbb{T} \times [0,T]. \end{equation} 
To conclude the proof of Theorem \ref{limitexistence}, we need to verify that our solution possesses the regularity specified in \eqref{solnrhoreg}-\eqref{solnpireg}. We now check that for any $\phi \in \mathcal{D}(\mathbb{T})$, the function
\begin{align*}
    t \mapsto \int_{\mathbb{T}} \rho(x,t) \phi(x)~dx
\end{align*}
is absolutely continuous on $[0,T]$. From \eqref{rhonegativesobolev} and the weak formulation of the continuity equation we have
\begin{align*}
    \frac{d}{dt} \int_{\mathbb{T}} \rho\phi~dx =  \int_{\mathbb{T}} \rho u\phi'~dx \in L^{1}(0,T).
\end{align*}
Therefore we are able to say that
\begin{align*}
    &\rho \in C_{weak}(0,T; L^{p}(\mathbb{T}))
\end{align*} for $p \in [1, \infty)$. We finally observe that since $\| \partial_{t}w_{n}\|_{L^{1}_{t,x}} + \|w_{n}\|_{L^{\infty}_{t}W^{1,\frac{4}{3}}_{x}} \le C$ we indeed have $w \in C(0,T; W^{1,\frac{4}{3}}(\mathbb{T})) \hookrightarrow C([0,T] \times \mathbb{T})$. This completes the proof. \end{proof}
\begin{appendices}
\section{Completing the blow-up lemma} 
In Section 2, we showed that
  \begin{equation} \label{buc}
    \sup_{t \in [0,T^{*})} \|\rho_{n}\|_{L^{\infty}(0,t; H^{2})} + \sup_{t \in [0,T^{*})} \|u_{n}\|_{L^{\infty}(0,t; H^{2})} + \sup_{t \in [0,T^{*})} \|u_{n}\|_{L^{2}(0,t; H^{3})} < + \infty.
\end{equation} In other words, we completed the base step of induction for Lemma \ref{blowup}. Recall that for a given integer $k \ge 0$ we also introduced the notation
\begin{equation*}
    M_{k}^{n} \equiv M_{k}^{n}(t, \|\rho_{n}\|_{L^{\infty}_{t}H^{k}_{x}}, \|u_{n}\|_{L^{\infty}_{t}H^{k}_{x}}, \|u_{n}\|_{L^{2}_{t}H^{k+1}_{x}}, \gamma_{n}, E_{n},  \overline{\rho_{n}}, (\underline{\rho_{n}})^{-1}),
\end{equation*} to denote an arbitrary function that is increasing with respect to the specified arguments, and which also satisfies $\sup_{t \in [0,T^{*})} M_{k}^{n} < + \infty$. The inductive step corresponds to the following proposition:
 \begin{proposition}
     Suppose $(\rho_{n}, u_{n})$ is a regular solution to \eqref{a1}-\eqref{a2} on $[0,T^{*})$ with initial data $(\rho_{n}^{0}, u_{n}^{0}) \in H^{k}(\mathbb{T}) \times H^{k}(\mathbb{T})$, where $n \in \mathbb{N}^{+}$ is fixed and $k > 2$. We assume that
     \begin{equation} \label{Mk1}
         \|\rho_{n}\|_{L^{\infty}_{t}H^{k-1}_{x}} + 
         \|u_{n}\|_{L^{\infty}_{t}H^{k-1}_{x}} + \|u_{n}\|_{L^{2}_{t}H^{k}_{x}} \le M_{k-2}^{n},
     \end{equation} Then,
     \begin{equation} \label{Mk4}
         \|\rho_{n}\|_{L^{\infty}_{t}H^{k}_{x}} + 
         \|u_{n}\|_{L^{\infty}_{t}H^{k}_{x}} + \|u_{n}\|_{L^{2}_{t}H^{k+1}_{x}} \le M_{k-1}^{n}.
     \end{equation}
 \end{proposition}
 \begin{proof} \label{Vinduction}
     Taking $l = k-1$ in \eqref{HOR2}, we have 
     \begin{align*}
        \frac{1}{2}\frac{d}{dt}\int_{\mathbb{T}} |\partial_{x}^{k-1}V_{n}|^{2}~dx &= - \int_{\mathbb{T}} \partial_{x}^{k-1} \left( (u_{n} + \frac{\lambda_{n}(\rho_{n})}{\rho_{n}^{2}}\partial_{x}\rho_{n}) \partial_{x}V_{n}\right) \partial_{x}^{k-1}V_{n}~dx + \int_{\mathbb{T}} \partial_{x}^{k-1} \left( \frac{\lambda_{n}(\rho_{n})}{\rho_{n}}\partial_{x}^{2}V_{n}\right)\partial_{x}^{k-1}V_{n}~dx \\[1ex]  &- \int_{\mathbb{T}} \partial_{x}^{k-1} \left( \frac{\lambda_{n}'(\rho_{n})\rho_{n} + \lambda_{n}(\rho_{n})}{(\lambda_{n}(\rho_{n}))^{2}}V_{n}^{2}\right)\partial_{x}^{k-1}V_{n}~dx.
     \end{align*} Integrating by parts and introducing commutator notation,
     \begin{align} \notag
         \frac{1}{2}\frac{d}{dt}&\int_{\mathbb{T}} |\partial_{x}^{k-1}V_{n}|^{2}~dx + \int_{\mathbb{T}} \frac{\lambda_{n}(\rho_{n})}{\rho_{n}}|\partial_{x}^{k}V_{n}|^{2}~dx  \\[1ex] \notag
     = & \int_{\mathbb{T}}  \left(u_{n} + \frac{\lambda_{n}(\rho_{n})}{\rho_{n}^{2}}\partial_{x}\rho_{n}\right) \partial_{x}^{k}V_{n} ~ \partial_{x}^{k-1}V_{n}~dx + \int_{\mathbb{T}} \left[ \partial_{x}^{k-2},~ u_{n} + \frac{\lambda_{n}(\rho_{n})}{\rho_{n}^{2}}\partial_{x}\rho_{n} \right] \partial_{x}V_{n} ~\partial_{x}^{k}V_{n}~dx \\[1ex] \notag & - \int_{\mathbb{T}} \left[\partial_{x}^{k-2}, \frac{\lambda_{n}(\rho_{n})}{\rho_{n}} \right] \partial_{x}^{2}V_{n}~ \partial_{x}^{k}V_{n}~dx - \int_{\mathbb{T}} \left(\frac{\lambda_{n}'(\rho_{n})\rho_{n} + \lambda_{n}(\rho_{n})}{(\lambda_{n}(\rho_{n}))^{2}}V_{n} \right) |\partial_{x}^{k-1}V_{n}|^{2}~dx \\[1ex] &- \int_{\mathbb{T}} \left[\partial_{x}^{k-1}, \frac{\lambda_{n}'(\rho_{n})\rho_{n} + \lambda_{n}(\rho_{n})}{(\lambda_{n}(\rho_{n}))^{2}}V_{n} \right] V_{n}\partial_{x}^{k-1}V_{n}~dx =: \sum_{n=1}^{5}I_{n}. \label{bu0}
     \end{align} We now estimate $I_{1} - I_{5}$.  Integrating by parts,
     \begin{align*}
         I_{1} &\le \left| \int_{\mathbb{T}} (\partial_{x}^{k-1}V_{n})^{2}(\partial_{x}u_{n} + \partial_{x} (\frac{\lambda(\rho_{n})}{\rho_{n}^{2}}\partial_{x}\rho_{n}))~dx\right| \\[1ex] &\le \|\partial_{x}^{k-1}V_{n}\|^{2}_{L^{2}_{x}}\left\{ \|\partial_{x}u_{n}\|_{L^{\infty}_{x}} + \|\partial_{x} (\rho_{n}^{-2}\lambda_{n}(\rho_{n}))\partial_{x}\rho_{n} \|_{L^{\infty}_{x}} + \| \rho_{n}^{-2}\lambda(\rho_{n})\partial_{x}^{2}\rho_{n}\|_{L^{\infty}_{x}} \right\} \\[1ex] &\le M_{k-1}^{n}\|\partial_{x}^{k-1}V_{n}\|^{2}_{L^{2}_{x}}.
     \end{align*} Noticing that $\frac{\lambda_{n}(\rho_{n})}{\rho_{n}^{2}}\partial_{x} \rho_{n} = \partial_{x} p_{n} (\rho_{n})$ and using the well-known Kato-Ponce commutator estimates \cite{KatoPonce},
     \begin{align*}
         I_{2} &= \int_{\mathbb{T}} \left[ \partial_{x}^{k-2}, u_{n}+\partial_{x} p_{n} (\rho_{n}) \right]\partial_{x}V_{n}~\partial_{x}^{k-1}V_{n}~dx \le \left\|\left[ \partial_{x}^{k-2}, u_{n}+\partial_{x} p_{n} (\rho_{n}) \right]\partial_{x}V_{n} \right\|_{L^{2}_{x}} \| \partial_{x}^{k}V_{n}\|_{L^{2}_{x}} \\[1ex] &\le
         \left(\|\partial_{x}^{k-2}(u_{n} + \partial_{x} p_{n} (\rho_{n}))\|_{L^{2}_{x}}\|\partial_{x}V_{n}\|_{L^{\infty}_{x}} + \|\partial_{x} (u_{n} + \partial_{x} p_{n} (\rho_{n}) )\|_{L^{\infty}_{x}} \left\| \partial_{x}^{k-1}V_{n}\right\|_{L^{2}_{x}}\right) \|\partial_{x}^{k}V_{n}\|_{L^{2}_{x}} \\[1ex] &\le M_{k-1}^{n} + M_{k-1}^{n}\|\partial_{x}^{k-1}V_{n}\|^{2}_{L^{2}_{x}} +  M_{k-1}^{n}\|\partial_{x}^{k-1}\rho_{n}\|_{L^{2}_{x}}^{2} + \epsilon \|\partial_{x}^{k}V_{n}\|_{L^{2}_{x}}^{2} \\[1ex] &\le M_{k-1}^{n} + M_{k-1}^{n}\|\partial_{x}^{k-1}V_{n}\|^{2}_{L^{2}_{x}} + \epsilon_{1} \|\partial_{x}^{k}V_{n}\|_{L^{2}_{x}}^{2}.
     \end{align*} Next,
     \begin{align*}
         I_{3} &\le \left| \int_{\mathbb{T}} \left[\partial_{x}^{k-2}, \frac{\lambda_{n}(\rho_{n})}{\rho_{n}} \right] \partial_{x}^{2}V_{n}~ \partial_{x}^{k}V_{n}~dx \right| \le \left\| \left[\partial_{x}^{k-2}, \frac{\lambda(\rho_{n})}{\rho_{n}} \right] \partial_{x}^{2}V_{n}\right\|_{L^{2}_{x}}\|\partial_{x}^{k}V_{n}\|_{L^{2}_{x}} \\[1ex] &\le \left( \left\|\partial_{x}^{k-2} (\rho_{n}^{-1}\lambda(\rho_{n}) ) \right\|_{L^{2}_{x}} \|\partial_{x}^{2}V_{n}\|_{L^{\infty}_{x}} + \|\partial_{x} (\rho_{n}^{-1}\lambda(\rho_{n}))\|_{L^{\infty}_{x}}\|\partial_{x}^{k-1}V_{n}\|_{L^{2}_{x}} \right)\|\partial_{x}^{k}V_{n}\|_{L^{2}_{x}} \\[1ex] &\le M_{k-1}^{n} + \epsilon_{2} \|\partial_{x}^{k}V_{n}\|^{2}_{L^{2}_{x}} +M_{k-1}^{n}\|\partial_{x}^{k-1}V_{n}\|^{2}_{L^{2}_{x}},
     \end{align*} It is also straightforward to verify that
     \begin{align*}
         I_{4} \le \left|\int_{\mathbb{T}} \left(\frac{\lambda_{n}'(\rho_{n})\rho_{n} + \lambda_{n}(\rho_{n})}{(\lambda_{n}(\rho_{n}))^{2}}V_{n} \right) |\partial_{x}^{k-1}V_{n}|^{2}~dx \right| \le M_{k-1}^{n}\|\partial_{x}^{k-1}V_{n}\|_{L^{2}_{x}}^{2}.
     \end{align*} Making use of commutator estimates once more, we have
     \begin{align} \notag
         I_{5} &\le \left|\int_{\mathbb{T}} \left[\partial_{x}^{k-1}, \frac{\lambda_{n}'(\rho_{n})\rho_{n} + \lambda_{n}(\rho_{n})}{(\lambda_{n}(\rho_{n}))^{2}}V_{n} \right] V_{n}\partial_{x}^{k-1}V_{n}~dx \right| = (\gamma_{n}+2)\left| \int_{\mathbb{T}} \left[\partial_{x}^{k-1},~ V_{n}(\lambda_{n}(\rho_{n}))^{-1} \right]V_{n} \partial_{x}^{k-1}V_{n}\right| \\[1ex] &\le (\gamma_{n}+2)\|\left[\partial_{x}^{k-1},~ V_{n}(\lambda_{n}(\rho_{n}))^{-1} \right]V_{n} \|_{L^{2}_{x}} \|\partial_{x}^{k-1}V_{n} \|_{L^{2}_{x}} \notag \\[1ex] 
         &\le (\gamma_{n}+2) \left( \| \partial_{x}^{k-1} (V_{n} (\lambda_{n}(\rho_{n}))^{-1})\|_{L^{2}_{x}}\|V_{n}\|_{L^{\infty}_{x}} + \| \partial_{x} (V_{n} (\lambda_{n}(\rho_{n}))^{-1}) \|_{L^{\infty}_{x}} \|\partial_{x}^{k-2} V_{n}\|_{L^{2}_{x}} \right)\|\partial_{x}^{k-1}V_{n} \|_{L^{2}_{x}}
         . \label{bu1}
     \end{align} We now note that
     \begin{equation*}
        \partial_{x}^{k-1}(V_{n}(\lambda_{n}(\rho_{n}))^{-1}) = \left[\partial_{x}^{k-1}, V_{n} \right](\lambda_{n}(\rho_{n}))^{-1} + \partial_{x}^{k-1}(\lambda_{n}(\rho_{n})^{-1})V_{n},
     \end{equation*} and moreover it is straightforward to deduce that
     \begin{align*}
         \| \left[\partial_{x}^{k-1}, V_{n} \right](\lambda_{n}(\rho_{n}))^{-1} \|_{L^{2}_{x}} &\le M_{k-1}^{n}\|\partial_{x}^{k-1}V_{n}\|_{L^{2}_{x}} + M_{k-1}^{n}\|\partial_{x}^{k-1}(\lambda_{n}(\rho_{n}))^{-1}\|_{L^{2}_{x}} \\[1ex]
         &\le M_{k-1}^{n}\|\partial_{x}^{k-1}V_{n}\|_{L^{2}_{x}} + M_{k-1}^{n}.
     \end{align*} Thus, we have
     \begin{equation*}
         \left\|\partial_{x}^{k-1} \left( V_{n}(\lambda_{n}(\rho_{n}))^{-1}) \right) \right\|_{L^{2}_{x}} \le M_{k-1}^{n} + M_{k-1}^{n}\|\partial_{x}^{k-1}V_{n}\|_{L^{2}_{x}}.
     \end{equation*} Going back to \eqref{bu1}, we then have
     \begin{align*}
         |I_{5}| &\le M_{k-1}^{n} + M_{k-1}^{n}\|\partial_{x}^{k-1}V_{n}\|_{L^{2}_{x}} + M_{k-1}^{n}\|\partial_{x}^{k-2}V_{n}\|_{L^{2}_{x}} \\[1ex]
         &\le M_{k-1}^{n} + M_{k-1}^{n}\|\partial_{x}^{k-1}V_{n}\|_{L^{2}_{x}}^{2}.
     \end{align*} Returning to \eqref{bu0} and using our newfound estimates for $I_{1} - I_{5}$, we have
     \begin{align*}
         \frac{1}{2}\frac{d}{dt}&\int_{\mathbb{T}} |\partial_{x}^{k-1}V_{n}|^{2}~dx + \int_{\mathbb{T}} \frac{\lambda_{n}(\rho_{n})}{\rho_{n}}|\partial_{x}^{k}V_{n}|^{2}~dx  \\[1ex]
         &\le M_{k-1}^{n}\|\partial_{x}^{k-1}V_{n}\|^{2}_{L^{2}_{x}} + M_{k-1}^{n}\|\partial_{x}^{k-1}V_{n}\|^{2}_{L^{2}_{x}} \\[1ex] &+ (\epsilon_{1}+\epsilon_{2}) \|\partial_{x}^{k}V_{n}\|^{2}_{L^{2}_{x}} +M_{k-1}^{n}\|\partial_{x}^{k-1}V_{n}\|^{2}_{L^{2}_{x}} + M_{k-1}^{n}\|\partial_{x}^{k-1}V_{n}\|_{L^{2}_{x}}^{2} + M_{k-1}^{n}\|\partial_{x}^{k-1}V_{n}\|_{L^{2}_{x}}^{2}+ M_{k-1}^{n} \\[1ex] &\le  M_{k-1}^{n} + (\epsilon_{1}+\epsilon_{2})\|\partial_{x}^{k}V_{n}\|^{2}_{L^{2}_{x}} + M_{k-1}^{n}\|\partial_{x}^{k-1}V_{n}\|_{L^{2}_{x}}^{2}.
     \end{align*} Thanks to the existence of the lower bound for $\rho_{n}$ on $[0,T^{*})$, we have that $\rho_{n}^{-1} \lambda_{n}(\rho_{n}) = \gamma_{n} \rho_{n}^{\gamma_{n}} \ge \gamma_{n} \underline{\rho_{n}}^{\gamma_{n}}$ and thus 
     \begin{align}
            \frac{1}{2}\frac{d}{dt}&\int_{\mathbb{T}} |\partial_{x}^{k-1}V_{n}|^{2}~dx + C \int_{\mathbb{T}} |\partial_{x}^{k}V_{n}|^{2}~dx  \le  M_{k-1}^{n}  + M_{k-1}^{n}\|\partial_{x}^{k-1}V_{n}\|_{L^{2}_{x}}^{2}, \label{bu4}
  \end{align} where $\epsilon_{i} > 0$ are chosen so that $C > 0$. This choice may depend on $n$. In particular, we have
     \begin{align*}
         \frac{d}{dt}\|\partial_{x}^{k-1}V_{n}\|^{2}_{L^{2}_{x}} \le M_{k-1}^{n}  + M_{k-1}^{n}\|\partial_{x}^{k-1}V_{n}\|_{L^{2}_{x}}^{2}. 
     \end{align*} An application of Gronwall's inequality gives us
     \begin{equation} \label{dxk-1V}
         \|\partial_{x}^{k-1}V_{n}\|_{L^{\infty}_{t}L^{2}_{x}} + \|\partial_{x}^{k}V_{n}\|_{L^{2}_{t,x}} \le M_{k-1}^{n}.
     \end{equation}
     Next, we attempt to estimate $\partial_{x}^{k}u_{n}$ and $\partial_{x}^{k+1}u_{n}$. Using the expression $V_{n} = \lambda_{n}(\rho_{n})\partial_{x}u_{n}$ we have
     \begin{equation*}
         \partial_{x}^{k}u_{n} = \partial_{x}^{k-1}(\lambda_{n}^{-1}V_{n}) = \lambda_{n}^{-1}\partial_{x}^{k-1}V_{n} + \left[\partial_{x}^{k-1}, \lambda_{n}^{-1} \right]V_{n}. 
     \end{equation*} Then using commutator estimates,
     \begin{align} \label{dxku1}
         \|\partial_{x}^{k}u_{n}\|_{L^{2}_{x}} \le \|\lambda_{n}^{-1}\|_{L^{\infty}_{x}}\|\partial_{x}^{k-1}V_{n}\|_{L^{2}_{x}} + \|\partial_{x}^{k-1}(\lambda_{n}^{-1})\|_{L^{2}_{x}}\|V_{n}\|_{L^{\infty}_{x}} + \|\partial_{x}\lambda_{n}^{-1}\|_{L^{\infty}_{x}}\|\partial_{x}^{k-2}V_{n}\|_{L^{2}_{x}}.
     \end{align} It is important to mention that by induction and the Sobolev inequality \eqref{Sinf} one can show that for $j\ge 1$, \begin{equation} \label{lambdakderiv}
         \|\partial_{x}^{j}\lambda_{n}(\rho_{n})\|_{L^{\infty}_{t}L^{2}_{x}} + \|\partial_{x}^{j}\lambda_{n}(\rho_{n})^{-1}\|_{L^{\infty}_{t}L^{2}_{x}} \le M_{j-1}^{n} + M_{j-1}^{n}\|\partial_{x}^{j}\rho_{n}\|_{L^{\infty}_{t}L^{2}_{x}}.
     \end{equation} Then going back to \eqref{dxku1} and using \eqref{dxk-1V} and the inductive hypothesis \eqref{Mk1} ,
     \begin{align}
     \begin{aligned}
         \|\partial_{x}^{k}u_{n}\|_{L^{\infty}_{t}L^{2}_{x}} &\le M_{k-1}^{n} + M_{k-1}^{n}\|\partial_{x}^{k-1}\rho_{n}\|_{L^{\infty}_{t}L^{2}_{x}} \le M_{k-1}^{n}. \label{dxku}
          \end{aligned}
     \end{align} Let's turn our attention to $\partial_{x}^{k+1}u_{n}$. In a similar fashion,
     \begin{equation}
         \partial_{x}^{k+1}u_{n} = \lambda_{n}^{-1}\partial_{x}^{k}V_{n} + \left[ \partial_{x}^{k}, \lambda_{n}^{-1} \right]V_{n}
     \end{equation} and so repeating the commutator estimates in the previous argument leads to
     \begin{align*}
         \|\partial_{x}^{k+1}u_{n}\|_{L^{2}_{x}} &\le M_{k-1}^{n}\|\partial_{x}^{k}V_{n}\|_{L^{2}_{x}} + M_{k-1}^{n}\|\partial_{x}^{k}(\lambda_{n})^{-1}\|_{L^{2}_{x}} + M_{k-1}^{n}\|\partial_{x}^{k-1}V_{n}\|_{L^{2}_{x}}.
     \end{align*} Squaring and integrating in time,
     \begin{align}
     \begin{aligned}
          \|\partial_{x}^{k+1}u_{n}\|_{L^{2}_{t,x}}^{2} &\le M_{k-1}^{n}\|\partial_{x}^{k}V_{n}\|_{L^{2}_{t,x}} + M_{k-1}^{n}\|\partial_{x}^{k}\rho_{n}\|_{L^{2}_{t}L^{2}_{x}} + M_{k-1}^{n}\|\partial_{x}^{k-1}V_{n}\|_{L^{2}_{t,x}} \\[1ex]
          &\le M_{k-1}^{n} + M_{k-1}^{n}\|\partial_{x}^{k}\rho_{n}\|_{L^{2}_{t}L^{2}_{x}}. \label{dxk+1u}\end{aligned}
     \end{align}
     Note that we are not yet able to bound $\|\rho_{n}\|_{L^{2}_{t}H^{k}_{x}}$. We now recall \eqref{HOR1}, which reads:
    \begin{align}
        \frac{1}{2}\frac{d}{dt} \|\partial_{x}^{k} \rho_{n}\|_{L^{2}_{x}}^{2} \le C\left( \|\partial_{x}^{k}u_{n}\|_{L^{2}_{x}} \|\partial_{x}^{k}\rho_{n}\|_{L^{2}_{x}}^{2} + \|\rho_{n}\|_{L^{\infty}_{x}} \|\partial_{x}^{k}\rho_{n}\|_{L^{2}_{x}}\|\partial_{x}^{k+1}u_{n}\|_{L^{2}_{x}}\right). 
    \end{align}
  Thanks to \eqref{dxk+1u} and \eqref{dxku}, we have
  \begin{equation*}
      \frac{1}{2}\frac{d}{dt} \|\partial_{x}^{k} \rho_{n}\|_{L^{2}_{x}}^{2} \le M_{k-1}^{n} \|\partial_{x}^{k}\rho_{n}\|_{L^{2}_{x}}^{2}
  \end{equation*} and therefore Gronwall's inequality gives us
  \begin{equation} \label{dxkrho}
      \|\partial_{x}^{k}\rho_{n}\|_{L^{\infty}_{t}L^{2}_{x}} \le M_{k-1}^{n}.
  \end{equation} Returning to \eqref{dxk+1u} and using this estimate also yields
  \begin{equation}
      \|\partial_{x}^{k+1}u_{n}\|_{L^{2}_{t,x}} \le M_{k-1}^{n}.
  \end{equation}
     Adding \eqref{dxku}, \eqref{dxkrho} and \eqref{dxk+1u}, 
    \begin{equation*}
         \|\partial_{x}^{k}u_{n}\|_{L^{\infty}_{t}L^{2}_{x}} + \|\partial_{x}^{k}\rho_{n}\|_{L^{\infty}_{t}L^{2}_{x}} + \|\partial_{x}^{k+1}u_{n}\|_{L^{2}_{t,x}} \le M_{k-1}^{n}.
    \end{equation*}
     Combining this with the inductive hypothesis \eqref{Mk1}, we have shown that
     \begin{equation} \label{bufinal}
         \|\rho_{n}\|_{L^{\infty}_{t}H^{k}_{x}} + \|u_{n}\|_{L^{\infty}_{t}H^{k}_{x}} + \|u_{n}\|_{L^{2}_{t}H^{k+1}_{x}} \le M_{k-1}^{n}.
     \end{equation} Note that in all of our estimates the notation $M_{k-1}^{n}$ was used to denote a function of time satisfying $\sup_{t \in [0,T^{*})} M_{k-1}^{n} < + \infty$. Therefore the right-hand side of \eqref{bufinal} also satisfies this condition. The proof is now complete.
 \end{proof}
 We conclude with a proof of a generalised Poincare inequality, similar to Proposition 7.2 of \cite{HCL}.
 \begin{proposition} \label{poincare}
     There exists $C \equiv C(M,K) >0$ such that
     \begin{equation} \label{poincareineq}
         \|u\|_{L^{1}_{x}} \le C \left( \|\partial_{x}u\|_{L^{1}_{x}} + \int_{\mathbb{T}}r|u|~dx  \right)
     \end{equation} for any $u \in W^{1,1}(\mathbb{T})$ and any non-negative $r \in L^{2}(\mathbb{T})$ satisfying
     \begin{equation} \label{poincarerassumptions}
         0 < M \le \int_{\mathbb{T}} r~dx < + \infty, \qquad \int_{\mathbb{T}} r^{2}~dx \le K,
     \end{equation} for some constants $M, K > 0$.
 \end{proposition}
 \begin{proof}
     We proceed by contradiction. If the proposition is false, then there exist sequences $\{u_{n}\}_{n=1}^{\infty}$, $\{r_{n}\}_{n=1}^{\infty}$ such that $\|u_{n}\|_{L^{1}_{x}} = 1$ (after suitable renormalisation), each $r_{n}$ satisfies \eqref{poincarerassumptions} and
     \begin{equation} \label{poincare1}
         \|\partial_{x}u_{n}\|_{L^{1}_{x}} + \int_{\mathbb{T}}r_{n}|u_{n}|~dx \le \frac{1}{n}.
 \end{equation} The assumption \eqref{poincarerassumptions} implies that there exists $r \in L^{2}(\mathbb{T})$ such that up to a subsequence, $r_{n} \rightharpoonup r$ weakly in $L^{2}(\mathbb{T})$. It is also easy to see that $\|u_{n}\|_{W^{1,1}(\mathbb{T})} \le 2$. Therefore due to the compact embedding $W^{1,1} \hookrightarrow \hookrightarrow L^{\infty}$ we have in particular that $u_{n} \to u$ strongly in $L^{2}(\mathbb{T})$. From \eqref{poincare1} we infer that $\partial_{x}u_{n} \to 0$ strongly in $L^{1}(\mathbb{T})$. Thus we may argue that $u_{n} \to u$ strongly in $W^{1,1}(\mathbb{T})$ with $\partial_{x}u=0$. Passing to the limit in the inequality appearing in \eqref{poincare1} and using the weak and strong $L^{2}$ convergences  of $r_{n}$  and $u_{n}$ respectively, we find that 
 \begin{equation}
     \int_{\mathbb{T}} r~dx = 0,
 \end{equation} which contradicts \eqref{poincarerassumptions}.
 \end{proof}
\section*{Acknowledgements}
The work of M.A. Mehmood is supported by the EPSRC Early Career Fellowship no. EP/V000586/1. The author would also like to thank the anonymous referee for providing valuable feedback which helped to improve the quality of the manuscript.
\end{appendices}
\printbibliography

\end{document}